\documentclass[12pt]{amsart}

\usepackage[all]{xy}
\input xy
\xyoption{all}
\usepackage{amsmath}
\usepackage[mathscr]{eucal}
\usepackage{amssymb}
\usepackage{latexsym}
\usepackage{amsthm}
\usepackage{amscd}
\usepackage{txfonts}

\theoremstyle{plain}
\newtheorem{theorem}{Theorem}[section]
\newtheorem{proposition}[theorem]{Proposition}

\newtheorem{lemma}[theorem]{Lemma}

\newtheorem*{main}{{\bf Theorem}}

\theoremstyle{definition}
\newtheorem{remark}[theorem]{Remark}

\newcommand{\reffig}[1]{Figure~\ref{#1}}

\newcommand{\reflem}[1]{Lemma~{\rm \ref{#1}}}
\newcommand{\refprop}[1]{Proposition~{\rm \ref{#1}}}
\newcommand{\refthm}[1]{Theorem~{\rm \ref{#1}}}

\newcommand{\refsec}[1]{Section~{\rm \ref{#1}}}
\newcommand{\refsubsec}[1]{Subsection~{\rm \ref{#1}}}
\newcommand{\refsubsubsec}[1]{Subsubsection~{\rm \ref{#1}}}

\newcommand{\bA}{\mathbb{A}}

\newcommand{\bZ}{\mathbb{Z}}
\newcommand{\bF}{\mathbb{F}}
\newcommand{\bP}{\mathbb{P}}
\newcommand{\bC}{\mathbb{C}}

\newcommand{\cO}{\mathscr{O}}

\newcommand{\SL}{\mathrm{SL}}

\newcommand{\Div}{\mathrm{Div}\,}
\newcommand{\Pic}{\mathrm{Pic}\,}

\newcommand{\ord}{\mathrm{ord}\,}
\newcommand{\Tr}{\mathrm{Tr}\,}

\newcommand{\Cf}{\textrm{cf.}\;}

\newcommand{\hr}{\hookrightarrow}
\newcommand{\ra}{\rightarrow}
\newcommand{\dra}{\dashrightarrow}

\newcommand{\lan}{\langle}
\newcommand{\ran}{\rangle}

\newcommand{\fg}{\pi_1}

\newcommand{\isom}{\tilde{\ra}}

\newcommand{\opcit}{\textrm{op.\,cit.}}

\newcommand{\xra}[1]{\xrightarrow[]{#1}}

\newcommand{\vphi}{\varphi}

\renewcommand{\Im}{{\rm Im}}

\newcommand{\ga}{\alpha}
\newcommand{\gb}{\beta}

\newcommand{\geps}{\epsilon}

\newcommand{\ssm}{\smallsetminus}

\newcommand{\sing}{\mathrm{sing}}

\renewcommand{\div}{\mathrm{div}}

\title{Canonical components of character varieties
 of arithmetic two-bridge link complements}
\author{Shinya Harada}


\begin{document}

\pagestyle{myheadings}
\pagenumbering{arabic}

\maketitle

\markboth{Shinya Harada}{Canonical components of arithmetic two-bridge link groups}

\begin{abstract}
 The desingularizations of
 the canonical components of $\SL_2(\bC)$-character varieties of arithmetic two-bridge link groups are determined.
\end{abstract}

%%%%%%%%%%%%%%%%%%%%%%%%%%%%%%
%%%%%%%%%%%% Intro
%%%%%%%%%%%%%%%%%%%%%%%%%%%%%

\section{Introduction}
\label{Intro}

 The $\SL_2(\bC)$-character variety of a hyperbolic $3$-manifold
 is one of the central topics in the study of hyperbolic geometry.
 However little is known about the algebro-geometric
 properties of the character
 variety of a hyperbolic 3-manifold as an algebraic variety.
 In \cite{Vicente} the structure of
 the $\SL_2(\bC)$-character varieties 
 of torus knot groups were explicitly determined.
 In \cite{MPL}
 Macasieb, Petersen and van Luijk
 studied properties of
 the $\SL_2(\bC)$-character varieties of a 
 certain family of two-bridge knots which contains the twist knots.
 In fact they showed that
 the canonical components of
 the $\SL_2(\bC)$-character varieties of
 the twist knots are hyperelliptic curves.
 In \cite{EmilyLandes} Landes studied the canonical component of
 the Whitehead link complement.
% and she showed that
% its desingularization is isomorphic to
% the projective plane blown up at $10$ points.
%

 The Whitehead link complement is one of the examples of 
 arithmetic two-bridge links.
 In determining the canonical component of the character variety
 of the Whitehead link complement
 it was crucial that
 it can be considered as a (singular) conic bundle
% (with singular fibers)
 over the projective line $\bP^1:=\bP^1_{\bC}$
 in a specific projective space,
 which made it easy to obtain an explicit minimal model of
 the canonical component as an algebraic surface.
 It is already seen in other examples Landes computed
 that the canonical components of hyperbolic two-bridge links
 are not necessarily conic bundles over $\bP^1$ in general.

 It is known %\cite{MM1},
 (\cite{GMM2}) that
 there are only finitely many arithmetic two-bridge links
 in the $3$-sphere $S^3$.
 In fact, there are only $4$ such links, the figure $8$ knot $4_1=(5/3)$,
 the Whitehead link $5^2_1=(8/3)$, $6^2_2=(10/3)$
 and $6^2_3=(12/5)$ in the Rolfsen's table.
 The canonical component of the character variety of
 the figure 8 knot complement
 is well known, which is an elliptic curve (for instance,
 see \cite{W-Q}, Corollary 4.1).
 In this note we study the canonical components
 of the $\SL_2(\bC)$-character varieties of the other three arithmetic two-bridge links.
 (Unfortunately there was an error
 on the determination
 of a minimal model in the Whitehead link case
 in Landes' paper \cite{EmilyLandes},
 more specifically the proof of Corollary $1$ seems wrong,
 which was crucial for the determination of a minimal model
 in her paper.
 We will also recompute that in this note.
 Note that still the statement of Theorem $1$
 in her paper \cite{EmilyLandes} is true.)
 We can see that those also are (singular) conic bundles over $\bP^1$.
 Hence we can characterize their desingularizations
 by following the same method in \cite{EmilyLandes}.
 The main result in this note is as follows:

\begin{main}
 The desingularizations of
 the canonical components of the $\SL_2(\bC)$-character varieties
 of $5^2_1$, $6^2_2$ and $6^2_3$ are conic bundles over
 the projective line $\bP^1$ which are
 isomorphic to
% $\bP^1\times\bP^1$
% blown up at $9$, $12$ and $9$ points respectively.
%
 the surface obtained from $\bP^1\times\bP^1$
 by repeating a one-point blow up 
 $9$, $12$ and $9$ times
 (or equivalently
 obtained from $\bP^2$ by repeating a one-point
 blow up $10$, $13$ and $10$ times), respectively.
\end{main}

%\noindent
% Our motivation is to study the canonical components
% of these links arithmetically.
% For that purpose it is interesting to study the concrete
% description of the canonical components.
% You can see that, compared with the Whitehead link and $6^2_3$,
% the canonical component of the link $6^2_2$ has a different
% property: It has a symmetric structure on
% the two pair of variables.
% In the forthcoming paper
% we shall study these canonical components over
% algebraically closed fields with positive characteristic
% and zeta functions defined by them.
%
 Here we explain the outline of this note.
 In \refsec{sec:Natural models}
 we will show the explicit defining equations
 of the natural models
 ($\SL_2(\bC)$-character varieties)
 of the arithmetic two-bridge links
 and study their irreducible components.
 In particular we identify their canonical components.
 In \refsec{sec:ProjModel}
 we describe the singular points of certain projective models
 of the canonical components of the natural models
 which are equipped with the conic bundle
 structure over $\bP^1$.
 We also compute explicitly the degenerate fibers
 of them, which is useful for the determination of minimal
 models of the desingularizations of those projective models.
 In \refsec{sec:MinModel}
 we determine minimal models of the desingularization
 of our projective models
 by employing intersection theory of surfaces.
 In \refsec{sec:Desingularization}
 we characterize the desingularizations
 in terms of the number of blow ups from
 the minimal models obtained in \refsec{sec:MinModel}
 by computing the Euler characteristics of
 the projective models.

 The author is grateful to Kenji Hashimoto
 for his help on the determination of minimal models
 of the rational surfaces appeared as the canonical
 components of the $\SL_2(\bC)$-character varieties
 of the arithmetic two-bridge links in this note.
 He would also like to express his gratitude to the referee
 for his/her comments on the original manuscript,
 which drastically improved the composition of
 this note.

%%%%%%%%%%%%%%%%%%%%%%%%%%%%%%%%%%%%%%%%%%%
%%%%%%%%%%%%%%% section %%%%%%%%%%%%%%%%%%
%%%%%%%%%%%%%%%%%%%%%%%%%%%%%%%%%%%%%%%%%%%%

%\section{Details}

%%%%%%%%%%%%%%%%%%%%%%%%%%%%%%%%%%%%%%%%%%%%%
%%%%%%%%%%%% subsection %%%%%%%%%%%%%%%%%%%
%%%%%%%%%%%%%%%%%%%%%%%%%%%%%%%%%%%%%%%%%%%%

%\subsection{$\SL_2(\bC)$-Character variety}
%\label{subsec:char var}
\section{Natural models}
\label{sec:Natural models}

 The $\SL_2(\bC)$-character variety of a manifold
 $M$ is the set of characters of $\SL_2(\bC)$-representations of the fundamental group $\fg(M)$
 of $M$,
 which is known to be an affine algebraic set.
% (\cite{CS}).
 For basics and applications of $\SL_2(\bC)$-character
 varieties, see Culler and Shalen's original paper \cite{CS}
 or Shalen's survey paper \cite{ShalenHandbook}.
 It is known in general that
 we can compute the defining polynomials of
 $\SL_2(\bC)$-character varieties of finitely generated
 groups explicitly
 from their group presentations (\cite{GM}, Theorem 3.2).
 However in this note we only consider two-bridge link groups.
% which are generated by 2 elements.
 In this case we can compute the defining polynomials by
 Riley's method (\cite{RileyNonab}, \S$2$
%%%%
%%%% see also \cite{RileyHol}, lemma $7$  %%%%%%%%%
%%%%
 or see \cite{EmilyLandes}, \S$4$)
 by which we can compute the defining polynomials
 with less computation.

 Here we only show the result of computation
 of the defining polynomials of the $\SL_2(\bC)$-character varieties for the arithmetic two-bridge
 link groups and show which irreducible component
 is the canonical component
 (that is, the irreducible component containing
 the point corresponding to the holonomy representation).
 We also include the Whitehead link case
 for the convenience of the reader.
 For the detailed way of the computation
 of the defining polynomials,
 see \cite{EmilyLandes}, \S$4$.

%%%%%%%%%%%%%%%%%%%%%%%%%
%%%%%%%%%%%%%%%%%%%%%%
%%%%%%%%%%%%%%%%%%%%%%%%
\subsection{Preliminary:Notation}

 Here we summarize some basic results on
 group presentations of the fundamental group
 of two-bridge link groups.

 Let $L$ be a two-bridge link in the $3$-sphere.
 Then it is well-known
 (\Cf \cite{Burde}, Chapter $12$, G, E $12.1$) that
% (For the presentation of a two-bridge link group, see
% \cite{Burde}, Chapter $12$, G, E $12.1$).
%% or \cite{HiLoMo}, \S 5)
%
%
 its fundamental group has the following group presentation:
$$
 \fg(S^3 \ssm L) \, \isom\, \lan \; a, b \mid awAW = 1 \ran,
$$
 where
 $A$ and $B$ mean the inverses $a^{-1}$ and $b^{-1}$, respectively.
 When $L$ is represented by the Schubert's normal form
 $(\ga/\gb)$, the word $w$ is defined by
$$
 w:= b^{\geps_1}a^{\geps_2}b^{\geps_3}\cdots a^{\geps_{\ga-2}}b^{\geps_{\ga-1}},
$$
 where $\geps_i:=(-1)^{\left[\tfrac{i\gb}{\ga}\right]}$.
 Here, for a real number $r$,
 $[r]$ is the maximal integer not greater than $r$.

 The $\SL_2(\bC)$-character variety $X(M)$ of a manifold $M$
 is the set of $\SL_2(\bC)$-characters of $\fg(M)$,
 i.e., $X(M):=\left\{\chi_{\rho}:=\Tr(\rho):\fg(M) \ra \bC \mid \rho:\fg(M) \ra \SL_2(\bC) \right\}$.

 If $M$ is an orientable complete hyperbolic $3$-manifold,
 there is a special irreducible component containing
 the point corresponding to the character of the holonomy
 representation of $M$.
 It is known (\Cf \cite{ShalenHandbook}, Theorem $4.5.1$) that
 the canonical component of the $\SL_2(\bC)$-character
 variety of an $n$-component hyperbolic link complement has
 dimension $n$.
 Especially,
 the canonical component of the $\SL_2(\bC)$-character
 variety of a hyperbolic two-bridge link complement is
 an irreducible affine surface over $\bC$.

 Any $\SL_2(\bC)$-character $\chi$ of
 $\fg(S^3 \ssm L)$
% a two-generator group $\lan g,h \ran$
 is determined by the values
 $\chi(a),\chi(b),\chi(ab)$.
 Thus we have a canonical injection
 $X(S^3 \ssm L) \ra \bA^3:=\bC^3$
 defined by
$$
\chi_{\rho} \mapsto (x,y,z):=(\chi_{\rho}(a), \chi_{\rho}(b), \chi_{\rho}(ab)).
$$
\noindent
 Put $q:=x^2+y^2+z^2-xyz-4$.
 It is known (\Cf \cite{BookMaRe}, Lemma 1.2.3) that
 a representation $\fg(S^3 \ssm L) \ra \SL_2(\bC)$
 is reducible
 if and only if
 $q(x,y,z) = 0$.
 In particular the points corresponding to abelian characters
 are contained in the algebraic set $V(q)$.
 Here, for polynomials $f_1,\cdots,f_r \in \bC[x,y,z]$
 let $V(f_1,\cdots,f_r)$ be the set of
 common zeros of $f_1,\cdots,f_r$, i.e.
$$
 V(f_1,\cdots,f_r) := \{ (x,y,z) \in \bA^3 \mid f_i(x,y,z)=0,\; 1\le i \le r \}.
$$
%
% Two generator subgroup $\lan g,h \ran$
% of $\SL_2(\bC)$ is reducible
% if and only if $\Tr ([g,h])=2$,
% i.e., $Q(\Tr g, \Tr h) = 0$.
%

%\subsubsection{Whitehead link $5^2_1$ case}
\subsection{Whitehead link $5^2_1$ case}

% 5^2_1=(8/3)
%% The fundamental group $\fg(M_{\cW})$
% The fundamental group $\fg(S^3 \ssm 5^2_1)$
% of the Whitehead link complement $5^2_1=(8/3)$
% in the $3$-sphere $S^3$
% has the following group presentation:
%$$
% \fg(S^3 \ssm 5^2_1) \, \isom\, \lan \; a, b \mid awAW = 1 \ran,
%$$
% where $w:= baBABab$.
 The fundamental group $\fg(S^3 \ssm 5^2_1)$
 of the Whitehead link complement $5^2_1=(8/3)$
 has $w:= baBABab$.
 Then $\SL_2(\bC)$-character variety of $S^3 \ssm 5^2_1$ is defined by
 the following two polynomials
% The corresponding polynomials $F,G,H$ are written
% as follows:
%\begin{align*}
% F &= \tau_R - 2 = P^2 Q,\\
% G &= \tau_{Ra} - \tau_a = x P^2 Q,\\
% H &= \tau_{Rb} - \tau_b = (x-yz) PQ,
%\end{align*}
% where $P := -(z^3 - xyz^2 + (x^2 + y^2 -2)z -xy)$.
$$
%F:= P Q, \quad G:= P (y-2)(y+2),
 f_0:= p_0 q, \quad g_0:= p_0 (y-2)(y+2),
$$
 where $p_0 := z^3 - xyz^2 + (x^2 + y^2 -2)z -xy$
 and $q$ are irreducible in $\bC[x,y,z]$
 (for instance,
 if we assume there is a factorization of $p_0$,
 immediately we have a contradiction by comparing degrees
 of monomials on both sides.
 We can show that $q$ is irreducible in the same manner).
% by the similar argument
% in the $6^2_2$ case.
% and $Q:=x^2+y^2+z^2-xyz-4$.
 Hence we have the decomposition
$$
 X(S^3 \ssm 5^2_1) = V(p_0) \cup V(q,y-2) \cup V(q,y+2).
$$
 Here $V(q,y-2)$ and $V(q,y+2)$ are affine lines $\bA^1$.
 Now the affine algebraic set $V(p_0)$ defined by $p_0$ is the unique
 $2$-dimensional component of the character variety of $5^2_1$.
 Hence that is the canonical component $X_0(S^3 \ssm 5^2_1)$
 of $5^2_1$.
 The points of the algebraic set defined by the polynomial $q$
 correspond to the reducible representations
 of $\fg(S^3 \ssm 5^2_1)$.
% In particular,
% the affine algebraic set defined by $p_0$ contains the point
% corresponding to the holonomy representation of
% the Whitehead link complement.
% Hence it is the canonical component of the $\SL_2(\bC)$-character
% variety.
 We summerize that the natural model
 $X(S^3 \ssm 5^2_1)$ consists of three irreducible
 algebraic sets
 $V(p_0), V(q,y-2)$ and $V(q,y+2)$.
 The canonical component $X_0(S^3 \ssm 5^2_1)=V(p_0)$ is
 the unique irreducible algebraic subset
 of $X(S^3 \ssm 5^2_1)$ of dimension $2$.
% (hence the canonical component $X_0(S^3 \ssm 5^2_1)$ of
% $X(S^3 \ssm 5^2_1)$).
 The other two components consist of points
 corresponding to the reducible $\SL_2(\bC)$-characters
 of $\fg(S^3 \ssm 5^2_1)$.
%
%
%In [1]: M=Triangulation('5^2_1')
%
%In [2]: M
%
%Out[2]: L205001(0,0)(0,0)
%
%In [3]: M.num_tetrahedra()
%
%Out[3]: 4
%
%In [4]: M.gluing_equations()
%
%Out[4]: matrix([[ 1,  0,  0,  1,  0,  0,  1,  0,  0,  1,  0,  0],
%        [ 0,  2,  1,  0,  0,  1,  0,  0,  1,  0,  2,  1],
%        [ 0,  0,  1,  0,  2,  1,  0,  2,  1,  0,  0,  1],
%        [ 1,  0,  0,  1,  0,  0,  1,  0,  0,  1,  0,  0],
%        [ 1,  0,  0,  0,  0, -1,  0,  0,  0,  0, -1,  0],
%        [ 2, -1,  0,  1,  0, -1, -1,  0, -1,  0, -1,  0],
%        [ 0,  1,  0,  0,  0,  1,  0,  0,  0, -1,  0,  0],
%        [ 0,  1,  0, -1,  0,  1,  1,  0,  1, -2,  1,  0]])
%
%In [5]: M.gluing_equations(form='rect')
%
%Out[5]: [([1, 1, 1, 1], [0, 0, 0, 0], 1), ([-1, -1, -1, -1], [-1, 1, 1, -1], 1), ([-1, -1, -1, -1], [1, -1, -1, 1], 1), ([1, 1, 1, 1], [0, 0, 0, 0], 1), ([1, 1, 0, 0], [0, -1, 0, 1], -1), ([2, 2, 0, 0], [1, -1, -1, 1], 1), ([0, -1, 0, -1], [-1, 1, 0, 0], -1), ([0, -2, 0, -2], [-1, 1, 1, -1], 1)]
%
%

%\subsubsection{$6^2_2$ case}
\subsection{$6^2_2$ case}

% We compute the defining polynomials of the $\SL_2(\bC)$-character
% varieties of arithmetic two-bridge links
% from the typical group presentations of
% two-bridge link groups.
%
%
% The fundamental group of the arithmetic two-bridge link
% $6^2_2=(10/3)$ has the following group presentation:
%$$
% \fg(S^3 \ssm 6^2_2) \, \isom\, \lan \; a, b \mid awAW = 1 \ran,
%$$
% where $w:= babABAbab$.
 The fundamental group of the arithmetic two-bridge link
 $6^2_2=(10/3)$ has $w:= babABAbab$.
%
%
%% %$W=BABabaBAB$,
%
% $R=awAW=ababABAbabABABabaBAB=1$.
%
%   Hi.  It's SnapPy.  
%    SnapPy is based on the SnapPea kernel, written by Jeff Weeks.
%    Type "Manifold?" to get started.
%    
%In [1]: G=Manifold('6^2_2').fundamental_group()
%In [2]: G.meridian(0)
%Out[2]: 'bba'
%In [3]: G.meridian(1)
%Out[3]: 'AB'
%In [4]: G.relators()
%Out[4]: ['aabbbaBBAABBBAbb']
%In [5]: G.peripheral_curves()
%Out[5]: [('bba', 'abbbaa'), ('AB', 'ABBBAABBB')]
%In [6]: G.num_generators()
%Out[6]: 2
%In [7]: G.num_relators()
%Out[7]: 1
%
%
%   
%In [1]: M=Triangulation('6^2_2')
%
%In [2]: M
%
%Out[2]: L206002(0,0)(0,0)
%
%In [4]: M.gluing_equations()
%Out[4]: matrix([[ 1,  1,  0,  1,  1,  0,  0,  1,  0,  0,  1,  0],
%       [ 0,  1,  1,  0,  1,  1,  1,  0,  0,  1,  0,  0],
%        [ 0,  0,  1,  0,  0,  1,  1,  0,  1,  1,  0,  1],
%        [ 1,  0,  0,  1,  0,  0,  0,  1,  1,  0,  1,  1],
%        [ 0,  0,  0,  0,  0, -1,  0,  1,  0,  0,  0,  0],
%        [-1,  0,  0,  1,  0,  1,  0, -1, -1,  0,  0,  1],
%        [ 0,  0,  1,  0,  0,  0,  0, -1,  0,  0,  0,  0],
%        [ 0,  1,  0,  1,  0,  0, -1,  0,  0,  0,  0, -1]])
%
%In [5]: M.gluing_equations(form='rect')
%
%Out[5]: [([1, 1, 0, 0], [-1, -1, -1, -1], 1), ([-1, -1, 1, 1], [0, 0, 0, 0], 1), ([-1, -1, 0, 0], [1, 1, 1, 1], 1), ([1, 1, -1, -1], [0, 0, 0, 0], 1), ([0, 1, 0, 0], [0, -1, -1, 0], -1), ([-1, 0, 1, -1], [0, 1, 0, 1], -1), ([-1, 0, 0, 0], [1, 0, 1, 0], -1), ([0, 1, -1, 1], [-1, 0, 0, -1], -1)]
%
%In [6]: M.num_tetrahedra()
%
%Out[6]: 4
%
%
 Then the $\SL_2(\bC)$-character variety of $6^2_2$ is
 defined by the following two polynomials
$$
 f_1:= p_1 q, \quad g_1:= p_1 (y-2)(y+2),
$$
%
%\begin{align*}
%F:&=\tau_R-2 =P^2_1Q,\\
%
%G:&=\tau_{Ra}-\tau_a =x P^2_1 Q, \\
%
%H:&=\tau_{Rb}-\tau_b=-(xz-y)P_1Q.
%\end{align*}
%
 where $p_1:=z^4-xyz^3+(x^2+y^2-3)z^2-xyz+1$.
 Note that $p_1$ is irreducible in $\bC[x,y,z]$
 by the similar argument
 as in the Whitehead link case.
 The $\SL_2(\bC)$-character variety $X(S^3 \ssm 6^2_2)$
 consists of three algebraic sets
$$
 X(S^3 \ssm 6^2_2) = V(p_1) \cup V(q,y-2) \cup V(q, y+2).
$$
 Here $V(q,y-2)$ and $V(q,y+2)$ are affine lines $\bA^1$.
 Now the affine algebraic set $V(p_1)$ defined by $p_1$ is the unique
 $2$-dimensional component of the character variety of $6^2_2$.
 Hence that is the canonical component of $6^2_2$.

 Thus the natural model $X(S^3 \ssm 6^2_2)$
 consists of three irreducible components,
 the canonical component $X_0(S^3 \ssm 6^2_2) = V(p_1)$
 and two components $V(q,y-2)$ and $V(q, y+2)$
 which correspond to
 $\SL_2(\bC)$-reducible characters
% The algebraic set defined by the polynomial $q$
% corresponds to the reducible representations
 of $\fg(S^3 \ssm 6^2_2)$.

%\subsubsection{$6^2_3$ case}
\subsection{$6^2_3$ case}
%
% The fundamental group of the arithmetic two-bridge link
% $6^2_3=(12/5)$ has the group presentation
%$$
% \fg(S^3 \ssm 6^2_3) \, \isom\, \lan \; a, b \mid awAW = 1 \ran,
%$$
%% where $w:= AbaBAbabABab$ (\Cf \cite{DMM}, Proposition 2).
% where $w:=baBAbabABab$.
 The fundamental group of the arithmetic two-bridge link
 $6^2_3=(12/5)$ has $w:=baBAbabABab$.
\noindent
 Then $\SL_2(\bC)$-character variety of $6^2_3$ is defined
 by the following two polynomials
%$$ F_2:= P_2 Q (x^2+y^2+z^2-xyz-3), \quad
% G_2:= P_2 (x^2+y^2+z^2-xyz-3)(y-2)(y+2),
$$ f_2:= p_2 q r, \quad
 g_2:= p_2 r (y-2)(y+2),
$$
%
%
%\begin{align*}
%F:&=\tau_R - 2 \\
%  &=P_2^2 Q(x^2+y^2+z^2-xyz-3)^2,\\
%G:&=\tau_{Ra}-\tau_a =x F,\\ 
%H:&=\tau_{Rb}-\tau_b\\
%  &=-P_2Q(x^2+y^2+z^2-xyz-3)\\
%&\times (x^3-2x+xy^2-2x^2yz+3yz-y^3z+xy^2z^2+z^2x-z^3y).
%\end{align*}
\noindent 
 where $p_2:=z^3-xyz^2+(x^2+y^2-1)z-xy$
 and $r:=x^2+y^2+z^2-xyz-3$ are irreducible polynomials
 in $\bC[x,y,z]$
 by the similar argument
 as in the Whitehead link case.
%
% $Q$ is the same as in the $6^2_2$ case.
% $Q:=x^2+y^2+z^2-xyz-4$.
%
 Thus we have the decomposition
$$
 X(S^3 \ssm 6^2_3) = V(p_2) \cup V(r) \cup V(q,(y-2)(y+2)).
$$
 Here $V(q,(y-2)(y+2))$ is the union of two affine lines $\bA^1$.
 There are two irreducible components of dimension $2$
 in this case.
 By considering the hyperbolicity equations of $S^3 \ssm 6^2_3$
 and the fact that the images of meridians of holonomy
 representations are parabolic elements,
 we can compute the holonomy representation concretely.
 In fact, the holonomy representation of $S^3 \ssm 6^2_3$
 is defined by
$$
 a \mapsto
\begin{pmatrix} 1 & 1 \\ 0 & 1 \end{pmatrix}, \quad
 b \mapsto
\begin{pmatrix} 1 & 0 \\ \ga & 1 \end{pmatrix},
\text{ where } \ga = \dfrac{-1 + \sqrt{-7}}{2}
$$
 up to conjugation.
% where $r = \tfrac{-1 + \sqrt{-7}}{2}$.
 Then the point corresponding to the holonomy representation
 is $(x,y,z) = (2,2,2+\ga)$,
 which is a zero of the polynomial $p_2$.
 Hence the canonical component of $6^2_3$ is the irreducible component $V(p_2)$.
% $V(z^3-xyz^2+(x^2+y^2-1)z-xy)$.
% defined by $p_2$.
% $V(P_2)$.
 The other irreducible component of dimension $2$,
 $V(r)$ is a smooth affine cubic surface.
 Moreover, we see that
 its natural homogenization $V_+(R) \subset \bP^3$
 defined by $R:= (x^2+y^2+z^2)w + xyz -3w^3$ is
 a smooth projective cubic surface.
 A cubic surface in $\bP^3$ is a well-studied object.
 It is a Del Pezzo surface of degree $3$,
 which is isomorphic to $\bP^2$ with six points
 blown up
 (or $\bP^1\times\bP^1$ with five points blown up
 since $\bP^2$ with two points blown up is isomorphic to
 $\bP^1\times\bP^1$ with one point blown up,
 \Cf \cite{Hartshorne}, V, Remark 4.7.1).

 Thus the natural model $X(S^3 \ssm 6^2_3)$ consists of
 four irreducible components,
 two of which are
 the canonical component
 $X_0(S^3 \ssm 6^2_3) = V(p_2)$ and
 a smooth affine surface $V(r)$.
 The other two components correspond to
 reducible $\SL_2(\bC)$-characters of $\fg(S^3 \ssm 6^2_3)$.

%%%%%%%%%%%%%%%%%%%%%%%%%%%%%%%%%%%%%
%%%%%%%%%%%%%%%%%%%%%%%%%%%%%%%
%%%  figure 8 knot case
%%
%In [1]: M=Triangulation('4_1')
%
%In [2]: M
%
%Out[2]: L104001(0,0)
%
%In [3]: M.num_tetrahedra()
%
%Out[3]: 2
%
%In [4]: M.gluing_equations()
%
%Out[4]: matrix([[ 2,  1,  0,  2,  1,  0],
%        [ 0,  1,  2,  0,  1,  2],
%        [ 1,  0,  0,  0,  0, -1],
%        [ 1,  1,  1,  1, -1, -3]])
%
%In [5]: M.gluing_equations(form='rect')
%
%Out[5]: [([2, 2], [-1, -1], 1), ([-2, -2], [1, 1], 1), ([1, 1], [0, -1], -1), ([0, 4], [0, -2], 1)]
%
%
%%%%%%%%%%%%%%%%%%%%%%%%%%%%%%%%%%%%%%%%%%%%%%%%
%%%%%%%%%%%%%%%%%% subsection %%%%%%%%%%%%%%%%%%
%%%%%%%%%%%%%%%%%%%%%%%%%%%%%%%%%%%%%%%%%%%%%%%

%\subsection{Projective models of canonical components}
\section{Projective models}
\label{sec:ProjModel}

 Let $p_0$, $p_1$ and $p_2$ be the polynomials
 in the previous section
 which define
 the canonical component of
 the $\SL_2(\bC)$-character variety of
 the Whitehead link, $6^2_2$ and $6^2_3$ link respectively.
 Then the Jacobian criterion shows that
 $V(p_i)$ is a smooth affine surface for any $i$.
% The projective surfaces $V_+(p_i) \subset \bP^3$
 The projective surfaces in $\bP^3$
 obtained by homogenizing $p_i$ naturally
 have infinitely many singularities.
 Thus we consider another compactification,
% in a smaller projective space
 namely a compactification in $\bP^2\times\bP^1$
 to obtain a projective surface having less singularities.
 We follow the method introduced in
 \cite{EmilyLandes} and \cite{MPL}.
 After reviewing
 $A_n$-singularities in \refsubsec{subsec:An-singularities}
 we study the homogenizations $S_i$ of $V(p_i)$
 in $\bP^2\times\bP^1$.

\subsection{$A_n$-singularities}
\label{subsec:An-singularities}

 The du Val singularity (or rational double points) is
 one kind of isolated singularity of a complex surface
whose exceptional curve consists of a tree of
 rational smooth curves,
 which is the unique rational singularity
 for hypersurfaces in $\bA^3$.
 It is classified into three types
 ($A$-$D$-$E$ singularities).
 Here we only explain the $A_n$-singularity.
 The $A_n$-singularity is one type of the du Val singularity
 characterized by
 the singular point $(0,0,0)$ of
 the equation $x^2 + y^2 + z^{n+1} = 0$.
 The exceptional curve of $x^2 + y^2 + z^{n+1} = 0$
 at the singular point $(0,0,0)$
 obtained by blowing up some number of times
 consists of $n$ smooth projective irreducible curves
 (isomorphic to $\bP^1$)
 with self-intersection number $-2$,
 which intersects each other transversally described as
 \reffig{fig: exceptional curve A_n}.
% and \reffig{fig:exceptional curves of A_1A_2A_3}.

\begin{center}
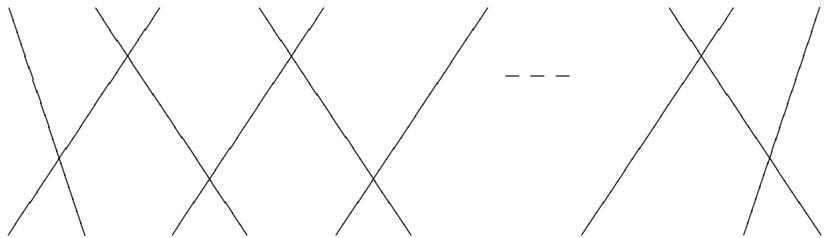
\begin{figure}[h]
%\centering
\quad\quad 
\xymatrix{
 \ar@{-}[dddr] & \ar@{-}[dddrr] &  & \ar@{-}[dddrr] & & & & & \ar@{-}[dddrr]       &  & \\
  &  &  &  &  & &\ar@{--}[r]& &  &  & \\
  &  &  &  &  & &           & &  &  & \\
 \ar@{-}[uuurr]&  & \ar@{-}[uuurr] &  & \ar@{-}[uuurr] & & & \ar@{-}[uuurr] &  & \ar@{-}[uuur] & 
%  &  &  &  &  & &           & &  &  &
}
\caption{Exceptional curve of $A_n$-singularity}
\label{fig: exceptional curve A_n}
\end{figure}
\end{center}

\noindent
 Each curve on both sides
 meets only with another curve at one point.
 The other curves meet with two other curves
 transversally.
% Thus it corresponds to the Dynkin diagram
 It is also expressed by the Dynkin diagram
\begin{figure}[h]
\quad\quad \!\!\!\!\!\!\!\!\!\!\!\!\!
\xymatrix{
 \circ \ar@{-}[r] & \circ \ar@{-}[r] & \circ\ar@{--}[r] & \circ \ar@{-}[r] & \circ
}
\caption{Dynkin diagram of $A_n$-singularity}
\label{fig:Dynkin diagram of $A_n$}
\end{figure}
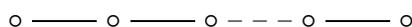

%
%\begin{figure}[h]
%\begin{center}
%\quad\quad
%\xymatrix{
%             & & & \ar@{-}[ddrr] & & & & \ar@{-}[dd]  & \ar@{-}[dd]&  \\
% \ar@{-}[rr] & & &         & & 	& \ar@{-}[rrr] & &            &  \\
%             & & & \ar@{-}[uurr] & & &               & &            & 
%}
%\caption{Exceptional curves of $A_1$, $A_2$, $A_3$}
%\label{fig:exceptional curves of A_1A_2A_3}
%\end{center}
%\end{figure}
%

%
% Here we summarize some properties of the topological
% Euler characteristic of a surface.
%
%% $\chi(X \cup Y) = \chi(X) + \chi(Y) - \chi(X \cap Y)$.
%% $\chi(\text{point}) = 1.
%

\noindent
 We have a relation between the topological Euler characteristics
 of a (singular) surface and its desingularization as follows.

\begin{lemma}[\Cf \cite{EmilyLandes}, Prop. 2]
\label{lem:EulerCharSmoothPt}
 Let $S$ be an irreducible smooth projective surface over $\bC$
 and $p$ a point of $S$.
 Let $\tilde{S}$ be the blow up of $S$ at $p$.
 Then $\chi(\tilde{S}) = \chi(S)+ 1$.
\end{lemma}

\begin{lemma}
\label{lem:EulerCharAnSingPt}
 Let $S$ be an irreducible projective surface over $\bC$,
 $p \in S$ an $A_n$-singular point
 and let $\epsilon:\tilde{S} \ra S$ be the desingularization
 of $S$ at the point $p$.
 Then we have $\chi(\tilde{S}) = \chi(S) + n$. 
\end{lemma}
\begin{proof}
 Note that the fiber of
 $\epsilon:\tilde{S} \ra S$ at the point $p$
 consists of $n$ projective lines
 which intersects with each other
 as in \reffig{fig: exceptional curve A_n}.
 Since $\chi(\bP^1) = 2$ and $\chi(\text{point}) = 1$,
 we have
\begin{align*}
 \chi(\tilde{S}) &= \chi(\tilde{S} \ssm \epsilon^{-1}(p)) + \chi(\epsilon^{-1}(p))
 = \chi(S \ssm \{p\}) + (n \chi(\bP^1) - (n-1)) \\
 &= \chi(S) -1 + n + 1 = \chi(S) + n.
\end{align*}
\end{proof}

\subsection{Projective models of the canonical components}

 Let
$$
\bP^2\times\bP^1:=\{(x:y:u,\;z:w)\mid (x:y:u)\in\bP^2,\;(z:w)\in\bP^1 \}
$$
 be the product of $\bP^2$ and $\bP^1$,
 and let
\begin{align*}
F_0&:=u^2z^3 -xyz^2w + (x^2+y^2-2u^2)zw^2 -xyw^3 \\
F_1&:=u^2z^4 - xyz^3w + (x^2+y^2-3u^2)z^2w^2 -xyzw^3 +u^2w^4 \\
F_2&:=u^2z^3-xyz^2w+(x^2+y^2-u^2)zw^2-xyw^3
\end{align*}
 be the homogenization of $p_0$, $p_1$ and $p_2$
 in $\bP^2\times\bP^1$.
 Consider the algebraic set
$$
S_i:=V(F_i):=\{(x:y:u,\;z:w)\in\bP^2\times\bP^1 \mid F_i(x,y,u,z,w)=0\}
$$
 defined by $F_i$ in $\bP^2\times \bP^1$.
 Since $\bA^3$ is naturally embedded in $\bP^2\times\bP^1$
 as
$$
\{(x:y:1,\;z:1)\mid (x,y)\in\bA^2,\;z\in\bA^1 \},
$$
 $V(p_i)$ is embedded in $S_i$ birationally.

 Let $\phi: \bP^2\times\bP^1 \ra \bP^1$ be
 the projection which is defined by
 $(x:y:u,\;z:w) \mapsto (z:w)$
 and define $\phi_i$ by the restriction of $\phi$ on $S_i$.
 We note that all the fibers of $S_i$ except
 finitely many points are smooth conic in $\bP^2$.
 Hence $\phi_i$ defines a (singular) conic bundle structure
 on $S_i$.
 In the following subsections
 we show the explicit description of the
 singular (degenerate) fibers of $\phi_i$
 and compute the Euler characteristic $\chi(\tilde{S}_i)$
 in terms of $\chi(S_i)$.

%%%%%%%%%%%%%%%%%%%%%%%%%%%%%%%%%%%%%%%%%%%%
%%%%%%%%%%% subsubsection %%%%%%%%%%%%%%%%%%
%%%%%%%%%%%%%%%%%%%%%%%%%%%%%%%%%%%%%%%%%

%\subsubsection{Whitehead link case}
\subsection{Whitehead link case}
\label{subsec:Whitehead sec:ProjModel}

 Let $F_0:=u^2z^3 -xyz^2w + (x^2+y^2-2u^2)zw^2 -xyw^3$
 be the homogenization of the polynomial
 $p_0=z^3 - xyz^2 + (x^2 + y^2 -2)z -xy$
 in the projective space $\bP^2 \times \bP^1$.
 Let $S_0:=V(F_0)$ be the algebraic set defined by $F_0$
 in $\bP^2\times\bP^1$.
 It is shown in \cite{EmilyLandes}, \S$4$ that
 $S$ has only four singular points
$$
(1:0:0,\; 1:0), \; (0:1:0,\; 1:0), \; (1:1:0,\; 1:1),\; (1:-1:0,\; 1:-1).
$$
 These four points are $A_1$ singularities
 (we can resolve the singularity by blowing up once)
 and the exceptional curves at the singular points
 are isomorphic to the projective line $\bP^1$.
%
% If we take the $\bP^2$-coordinate $(a:b:c)$ for the blow up
% at the singular points,
% the defining equations for the exceptional curves are
% as follows:
%\begin{align*}
% (1:0:0,\;1:0), &\phantom{aa} E=\{b^2-ac+c^2=0\}\;\isom\; \bP^1,\\
% (0:1:0,\;1:0), &\phantom{aa} E=\{a^2-bc+c^2=0\}\;\isom\; \bP^1,\\
% (1:1:0,\;1:1), &\phantom{aa} E=\{a^2-b^2-c^2=0\}\;\isom\;\bP^1,\\
%(1:-1:0,\;1:-1). &\phantom{aa} E=\{a^2-b^2-c^2=0\}\;\isom\; \bP^1.
%\end{align*}
%
%
\noindent
 Thus we have the following relation on the topological
 Euler characteristic of $S_0$ and the desingularization
 $\tilde{S_0}$ by \reflem{lem:EulerCharAnSingPt}:
$$
 \chi(\tilde{S_0}) = \chi(S_0\ssm S_{0,\sing}) + 4\chi(\bP^1)  = \chi(S_0) - 4 + 4\chi(\bP^1)=\chi(S_0) + 4.
$$
\noindent
 Here $S_{0,\sing}$ is the set of singular points of $S_0$.
 In \refsec{sec:Desingularization}
 we compute the topological Euler characteristic $\chi(S_0)$
 and determine $\tilde{S}_0$ in terms of
 the number of one-point blow ups from
 a minimal model of $\tilde{S}_0$.

\noindent
 Note that
 we can consider the surface $S_0$ (hence $\tilde{S}_0$)
 as a (singular) conic bundle over $\bP^1$ by
 the projection $\phi_0: S_0 \ra \bP^1$
 which is defined by
 $(x:y:u,\;z:w) \mapsto (z:w)$.
 It has six degenerate fibers at 
 $(1:0),\,(0:1),\,(1:\pm 1),\,(1:\pm \tfrac{1}{\sqrt{2}})$.
% In fact, $\phi^{-1}(1:0)$ is a double singular line,
% $\phi^{-1}(0:1)$ is the sum of two $\bP^1$ intersecting
% at one point,
% $\phi^{-1}(1:\pm 1)$ has two singular points
% at $(1:\pm 1: 0,\; 1:\pm 1)$
% and $\phi^{-1}(1:\pm \tfrac{1}{\sqrt{2}})$ have
% a singular point $(0:0:1,\; 1:\tfrac{1}{\sqrt{2}})$.
% Especially $\tilde{S} \ra S \xra{\phi} \bP^1$ have singular fibers.
%
%
% The surface $S_0$ (thus $\tilde{S}_0$ also)
% has $6$ degenerate fibers at 
% $(1:0),\,(0:1),\,(1:\pm 1),\,(1:\pm \tfrac{1}{\sqrt{2}})$.
 In fact, the degenerate fibers of $\phi_0:S_0\ra \bP^1$ are
 expressed as follows:
\begin{align*}
 \phi^{-1}_0(1:0)&\isom\{(x:y:u)\in \bP^2 \mid u^2=0\},\\
% \phi^{-1}(0:1)&\isom\{(x:y:u)\in \bP^2 \mid xy=0\}=\bP^1\vee\bP^1,\\
 \phi^{-1}_0(0:1)&\isom\{(x:y:u)\in \bP^2 \mid xy=0\},\\
% \phi^{-1}(1:\pm\tfrac{1}{\sqrt{2}})&=\{(x:y:u)\in \bP^2_k \mid
% \tfrac{1}{2}(x^2+y^2\mp \tfrac{3}{\sqrt{2}}xy=0\}\; \text{singular at }(0:0:1),\\
 \phi^{-1}_0(1:\pm\tfrac{1}{\sqrt{2}})&\isom\{(x:y:u)\in \bP^2 \mid
 \tfrac{1}{2}(x \mp \sqrt{2}y)(x \mp \tfrac{1}{\sqrt{2}}y)=0\},\\
%\phi^{-1}(1:\pm 1)&=\{(x:y:u)\in \bP^2_k \mid (x\mp y)^2-u^2=0\}.
\phi^{-1}_0(1:\pm 1)&\isom\{(x:y:u)\in \bP^2 \mid (x\mp y)-u)((x\mp y)+u)=0\}.
\end{align*}
\noindent
 Note that the fiber $\phi^{-1}_0(1:0)$ contains the singular points
 $(1:0:0,\; 1:0)$, $(0:1:0,\; 1:0)$ of the surface $S_0$.
 The fiber $\phi^{-1}_0(1:\pm 1)$
 contains the singular point $(1:\pm 1:0,\;1:\pm 1)$
 of the surface $S_0$ respectively.
%
% Note that the fibers
% $\phi^{-1}(1:\pm\tfrac{1}{\sqrt{2}})$ are
% irreducible curves which are singular at $(0:0:1,\;.1:\pm\tfrac{1}{\sqrt{2}})$ which are non-singular points of the surface $S_0$.
% Hence $\tilde{S} \ra S \xra{\phi} \bP^1$ have singular fibers.
%

%\subsubsection{$6^2_2$ case}
\subsection{$6^2_2$ case}

 Let $F_1:=u^2z^4 - xyz^3w + (x^2+y^2-3u^2)z^2w^2 -xyzw^3 +u^2w^4$
 be the homogenization of $p_1=z^4-xyz^3+(x^2+y^2-3)z^2-xyz+1$
 in $\bP^2\times\bP^1$ with coordinates $x,y,u$ and $z,w$, and
 let
$$
S_1:=V(F_1):=\{(x:y:u,\;z:w)\in\bP^2\times\bP^1 \mid F_1(x,y,u,z,w)=0\}
$$
 be the algebraic set defined by $F_1$
 in $\bP^2\times \bP^1$.
 (This is symmetric on $x$, $y$ and $z$, $w$.)
 This projective surface has only finitely many singular points.
 In fact, its singularities are only the following six points:
\begin{align*}
 (1:0:0,\;1:0),\;\,& (0:1:0,\;1:0),\; (1:0:0,\;0:1),\\
 (0:1:0,\;0:1),\;\,& (1:1:0,\; 1:1),\; (1:-1:0,\;1:-1),
\end{align*}
 which are $A_1$ singularities.
% (for singularities, see for instance \cite{GriHar}).
 Especially they are resolved by one blow up at each point.
 Let $\tilde{S}_1 \ra S_1$ be the desingularization of $S_1$
 blown up at these six points.
 Then the exceptional curve at each singular point is
 isomorphic to $\bP^1$.
%
%\begin{align*}
% (1:0:0,\;1:0),&\phantom{aa} E=\{b^2-ac+c^2=0\} \;\isom\; \bP^1,\\
% (0:1:0,\;1:0),&\phantom{aa} E=\{b^2-ac+c^2=0\} \;\isom\; \bP^1,\\
% (1:0:0,\;0:1),&\phantom{aa} E=\{b^2-ac+c^2=0\} \;\isom\; \bP^1,\\
% (0:1:0,\;0:1),&\phantom{aa} E=\{b^2-ac+c^2=0\} \;\isom\; \bP^1,\\
% (1:1:0,\;1:1).&\phantom{aa} E=\{a^2-b^2-c^2=0\}\;\isom\; \bP^1,\\
% (1:-1:0,\;1:-1).&\phantom{aa} E=\{a^2-b^2-c^2=0\}\;\isom\; \bP^1.
%\end{align*}
%
\noindent
 Thus we have the following relation on the topological
 Euler characteristic of $S_1$ and $\tilde{S}_1$
 by \reflem{lem:EulerCharAnSingPt}:
$$
 \chi(\tilde{S}_1) = \chi(S_1\ssm S_{1,\sing})+6\chi(\bP^1) = \chi(S_1) - 6 + 6\chi(\bP^1)=\chi(S_1) + 6.
$$
%
% Let $\vphi_i : S_i \ra \bP^1$ be the projection defined by
% $(x:y:u,\;z:w) \mapsto (z:w)$.
% Then we can consider
% the projective surfaces $S_i$
% as conic bundles over the projective line $\bP^1$
% by considering the projection $S_i \ra \bP^1$
% defined by $(x:y:u,\;z:w) \mapsto (z:w)$.
% A conic bundle over $\bP^1$ means a smooth algebraic surface $S$
% with a morphism $S \ra \bP^1$ whose fibers are conic
% (we allow the possibility of degenerate fibers).
% The following two propositions are important to compute the
% resolution of singularities of $S_i$.
%
%
\noindent
 As a (singular) conic bundle over $\bP^1$ by
 the projection $\phi_1: S_1 \ra \bP^1$
 the surface $S_1$ (hence $\tilde{S}_1$)
 has eight degenerate fibers at 
 $(1:0),\,(0:1),\,(1:\pm 1),\,(1:\pm \tfrac{\sqrt{5}\pm 1}{2})$.
% In fact, $\phi^{-1}(1:0)$ and $\phi^{-1}(0:1)$
% are double $\bP^1$ lines,
% are double singular lines,
% $\phi^{-1}(1:\pm 1)$ have a singular point $(1:\pm 1:0,\;1:\pm 1)$,
% and $\phi^{-1}(1:\pm \tfrac{\sqrt{5}\pm 1}{2})$ are
% singular curves with one point singularity at
% $(0:0:1,\;1:\pm \tfrac{\sqrt{5}\pm 1}{2})$.
% Thus $\tilde{S}_1 \ra \bP^1$ at least have singular fibers
% at $(1:\pm \tfrac{\sqrt{5}\pm 1}{2})$.
%
%
% The surface $S_1$ (and its desingularization $\tilde{S}_1$)
% has $8$ degenerate fibers at 
% $(1:0),\,(0:1),\,(1:\pm 1),\,(1:\pm \tfrac{\sqrt{5}\pm 1}{2})$.
%
 In fact, they are written as follows:
\begin{align*}
 \phi^{-1}_1(1:0)&\isom\{(x:y:u)\in \bP^2 \mid u^2=0\},\\
 \phi^{-1}_1(0:1)&\isom\{(x:y:u)\in \bP^2 \mid u^2=0\},\\
% \phi^{-1}(1:\tfrac{\sqrt{5}\pm 1}{2})&=\{(x:y:u)\in \bP^2_k \mid
% \tfrac{3\pm\sqrt{5}}{2}(x^2+y^2 - \sqrt{5}xy)=0\}\; \text{singular at }(0:0:1),\\
 \phi^{-1}_1(1:\tfrac{\sqrt{5}\pm 1}{2})&\isom\{(x:y:u)\in \bP^2 \mid
 \tfrac{3\pm\sqrt{5}}{2}(x - \tfrac{\sqrt{5}+1}{2}y)(x - \tfrac{\sqrt{5}-1}{2}y)=0\},\\
% \phi^{-1}(1:-\tfrac{\sqrt{5}\pm 1}{2})&=\{(x:y:u)\in \bP^2_k \mid
% \tfrac{3\pm\sqrt{5}}{2}(x^2+y^2 + \sqrt{5}xy)=0\}\; \text{singular at }(0:0:1),\\
 \phi^{-1}_1(1:-\tfrac{\sqrt{5}\pm 1}{2})&\isom\{(x:y:u)\in \bP^2 \mid
 \tfrac{3\pm\sqrt{5}}{2}(x + \tfrac{\sqrt{5}+1}{2}y)(x + \tfrac{\sqrt{5}-1}{2}y)=0\},\\
%\phi^{-1}(1:\pm 1)&=\{(x:y:u)\in \bP^2_k \mid (x\mp y)^2-u^2=0\}.
\phi^{-1}_1(1:\pm 1)&\isom\{(x:y:u)\in \bP^2 \mid ((x\mp y)-u)((x\mp y)+u)=0\}.
\end{align*}

\noindent
% Note that the fibers
% $\phi^{-1}(1:\pm\tfrac{\sqrt{5}\pm 1}{2})$ contains
% a singular point $(0:0:1,\;1:\pm\tfrac{\sqrt{5}\pm 1}{2})$
% which is a non-singular point of the surface $S_1$ over $k$.
 We remark that
 the fiber $\phi^{-1}_1(1:0)$
 (resp. $\phi^{-1}_1(0:1)$)
 contains the singular points
 $(1:0:0,\;1:0)$, $(0:1:0,\;1:0)$
 (resp. $(1:0:0,\;0:1)$, $(0:1:0,\;0:1)$) of $S_1$,
 and that each fiber $\phi^{-1}_1(1:\pm 1)$ contains
 the singular point $(1:\pm 1:0,\; 1:\pm 1)$ of $S_1$.
 In the next section
 we compute a minimal model of the desingularization
 of the surface $S_1$.

%%%%%%%%%%%%%%%%%%%%%%%%%%%%%%%%%%%%%%
%%%%%%%%%%%%% subsubsection %%%%%%%%%
%%%%%%%%%%%%%%%%%%%%%%%%%%%%%%%%%%%%%

%\subsubsection{$6^2_3$ case}
\subsection{$6^2_3$ case}

 Let $p_2:=z^3-xyz^2+(x^2+y^2-1)z-xy$.
% We see that $V(P_2) \subset \bA^3$ is an affine smooth surface.
 Let $F_2=u^2z^3-xyz^2w+(x^2+y^2-u^2)zw^2-xyw^3$ be the
 homogenization of $p_2$ in $\bP^2\times\bP^1$
 with the coordinates $(x:y:u,\;z:w)$.
 The corresponding projective surface
 $S_2=V(F_2) \subset \bP^2\times\bP^1$
 has the following four singular points:
$$
 (1:0:0,\;1:0),\; (0:1:0,\;1:0),\; (1:1:0,\;1:1),\; (1:-1:0,\;1:-1).
$$
 We remark that the first two points are $A_1$ singularities
 and the other two points are $A_3$ singularities.
 Hence we can resolve the first two singularities by blowing up once at each point but we have to blow up twice for the latter two points.
 Let $\tilde{S}_2$ be the smooth projective surface obtained
 by blowing up $S_2$ at these four singular points.
 The exceptional curves at the singular points
 $(1:0:0,\;1:0)$ and $(0:1:0,\;1:0)$ are isomorphic
 to $\bP^1$,
 and the exceptional curve at $(1:1:0,\;1:1)$ (resp. $(1:-1:0,\;1:-1)$)
% are the union of three curves $C_1\vee C_2\vee C_3$
 is the union of three curves $E^{+}_1$, $E^{+}_2$
 and $E^{+}_3$
 (resp. $E^{-}_1$, $E^{-}_2$ and $E^{-}_3$)
 respectively.
%
%
%\begin{align*}
% (1:0:0,\;1:0),&\phantom{aa} E=\{b^2-ac+c^2=0\} \;\isom\; \bP^1,\\
% (0:1:0,\;1:0),&\phantom{aa} E=\{b^2-ac+c^2=0\} \;\isom\; \bP^1,\\
%% (1:1:0,\;1:1),&\phantom{aa} E=\{a^2-c^2=0\} \;\isom\; \bP^1\vee \bP^1 \text{ intersecting at one point } (0:1:0),\\
% (1:1:0,\;1:1),&\phantom{aa} E \;\isom\; C_1\vee C_2\vee C_3,\\
%% \text{ intersecting at one point } (0:1:0),\\
%%(1:-1:0,\;1:-1).&\phantom{aa} E=\{a^2-c^2=0\} \;\isom\; \bP^1\vee \bP^1 \text{ intersecting at one point } (0:1:0).
%(1:-1:0,\;1:-1).&\phantom{aa} E \;\isom\; C_1\vee C_2\vee C_3.
%% \text{ intersecting at one point } (0:1:0).
%\end{align*}
%
\noindent
% Here $C_i$ are irreducible curves
 Here $E^{\pm}_i$ are smooth projective curves
 isomorphic to $\bP^1$ with self-intersection number $-2$.
% The curve $C_2$ intersects with $C_1$ and $C_3$ at one point respectively
% and $C_1$ and $C_3$ do not intersect each other.
 The curve $E^{\pm}_2$ intersects with $E^{\pm}_1$ and $E^{\pm}_3$ at one point respectively
 and $E^{\pm}_1$ and $E^{\pm}_3$ do not intersect each other
 (see \reffig{fig: exceptional curve A_n}).
 Thus we can express the Euler characteristic of $\tilde{S}_2$
 in terms of that of $S_2$ by \reflem{lem:EulerCharAnSingPt}:
$$
% \chi(\tilde{S}_2) = \chi(S_2 \ssm S_{2,\sing}) + 2\chi(\bP^1) + 2 \chi(\bP^1 \vee \bP^1)
 \chi(\tilde{S}_2) = \chi(S_2 \ssm S_{2,\sing}) + 2\chi(\bP^1) + 2% \chi(\bP^1 \vee \bP^1 \vee \bP^1)
% \chi(C_1\vee C_2\vee C_3)
 \chi(E^{+}_1\vee E^{+}_2\vee E^{+}_3)
% = 6 - 4 + 4 + 6 = 12.
 = \chi(S_2) - 4 + 4 + 8 = \chi(S_2) + 8.
$$
\noindent
% Note that
 We can consider the surface $S_2$ (hence $\tilde{S}_2$)
 as a conic bundle over $\bP^1$ by
 the projection $\phi_2: S_2 \ra \bP^1$.
 It has four degenerate fibers at 
 $(1:0),\,(0:1),\,(1:\pm 1)$.
%% In fact, $\phi^{-1}(1:0)$ is double $\bP^1$ line,
% In fact, $\phi^{-1}(1:0)$ is a double singular line,
% $\phi^{-1}(0:1)$ is the sum of two $\bP^1$ intersecting
% at one point, and $\phi^{-1}(1:\pm 1)$ are singular curves with
% singular locus $\{(x:\pm x:u,\; 1:\pm 1)\}$.
%% which are isomorphic to $\bP^1$.
%
% The degenerate fibers of the projection $\phi_2:S_2 \ra \bP^1$
% are the following:
% The surface $S_2$ has $4$ degenerate fibers at 
% $(1:0),\,(0:1),\,(1:\pm 1)$.
% In fact, $\phi^{-1}(1:0)$ is a double singular line,
% $\phi^{-1}(0:1)$ is the sum of two $\bP^1$ intersecting
% at one point, and $\phi^{-1}(1:\pm 1)$ are singular curves with
% singular locus $\{(x:\pm x:u,\; 1:\pm 1)\}$.
%% which are isomorphic to $\bP^1$.
% Hence $\tilde{S}_2 \ra \bP^1$ still have singular fibers.
\noindent
 In fact, the degenerate fibers of $\phi_2:S_2\ra \bP^1$ are
 expressed as follows:
\begin{align*}
 \phi^{-1}_2(1:0)&\isom\{(x:y:u)\in \bP^2 \mid u^2=0\},\\
% \phi^{-1}_2(0:1)&\isom\{(x:y:u)\in \bP^2 \mid xy=0\}=\bP^1\vee\bP^1,\\
 \phi^{-1}_2(0:1)&\isom\{(x:y:u)\in \bP^2 \mid xy=0\},\\
\phi^{-1}_2(1:\pm 1)&\isom\{(x:y:u)\in \bP^2 \mid (x\mp y)^2=0\}.
\end{align*}
\noindent
 We note that
 the fiber $\phi^{-1}_2(1:0)$ contains the singular points
 $(1:0:0,\; 1:0)$, $(0:1:0,\; 1:0)$ of $S_2$.
 The fiber $\phi^{-1}_2(1:\pm 1)$
 contains the singular point $(1:\pm 1:0,\;1:\pm 1)$
 of the surface $S_2$ respectively.

% Note that the fibers
% $\phi^{-1}(1:\pm\tfrac{1}{\sqrt{2}})$ are
% irreducible curves which are singular at $(0:0:1,\;.1:\pm\tfrac{1}{\sqrt{2}})$ which are non-singular points of the surface $S_0$.
% Hence $\tilde{S} \ra S \xra{\phi} \bP^1$ have singular fibers.
%

%%%%%%%%%%%%%%%%%%%%%%%%%%%%%%%%%%%%
%%%%%%%%%%%% minimal model %%%%%%%%
%%%%%%%%%%%%%%%%%%%%%%%%%%%%%%%%%%

%\subsection{Minimal models of the projective models of $S_i$}
%\label{subsec:MinModel}
\section{Minimal models}
\label{sec:MinModel}

 Since all the three surfaces $S_0$, $S_1$, $S_2$ are
 rational surfaces,
 their minimal models are either the projective plane $\bP^2$
 or the Hirzebruch surfaces $\bF_n$ ($n\ge 0$, $n \neq 1$)
 (\Cf \cite{Beauville}, Theorem V.10).
 Here we compute a minimal model of the surface $S_i$
 for each $i$,
 which is obtained naturally
 from its fibered surface structure.
 The purpose of this section is
 to prove the following two lemmas.

\begin{lemma}
\label{lem:geom.ruled}
 For each $\tilde{S}_i$
 we can blow down $\tilde{S}_i$ over $\bP^1$
 some number of times so that
 it becomes a geometrically ruled surface $T_i$ over $\bP^1$,
 namely all the fibers are isomorphic to $\bP^1$.
\end{lemma}

\begin{lemma}
\label{lem:determ.ofT_i}
 $T_i$ is isomorphic to $\bP^1\times\bP^1$.
\end{lemma}

\noindent
 In \refsubsec{subsec:preliminary on alg surface}
 we review some terminology on algebraic surfaces
 and some basic facts on the intersection theory
 of algebraic surfaces and minimal models.
 In \refsubsec{subsec:proof of first lem}
 and \refsubsec{subsec:proof of second lem}
 we show \reflem{lem:geom.ruled}
 and \reflem{lem:determ.ofT_i}.

%
% S \ra P^1 
%  C=P^1 \ra S global section 
%% C\dot C = deg N(S/C) degree of the normal bundle
%%%
%
% adjunction formulas
%
% N_{S/C}=O_S(C)\mid_C
%
% K_C= (K_S \otimes N_{S/C})\mid_C
%
%

%%%%%%%%%%%%%%%%%%%%%%%%%%%%%%%%%%%
%%%%%%%%%%%%% preliminary %%%%%%%
%%%%%%%%%%%%%%%%%%%%%%%%%%%%%%%%

\subsection{Preliminary on algebraic surfaces}
\label{subsec:preliminary on alg surface}

\subsubsection{Basic properties of intersection theory of surfaces}
\label{subsubsec:Intersection theory}

 Here we summarize some basic properties
 of the intersection theory of algebraic surfaces.
 In particular we include some results on
 the intersection numbers of divisors
 of fibered surfaces,
 which are necessary for the computation
 of minimal models obtained by
 blowing down the conic bundles over $\bP^1$
 appeared in the previous section.
 For more details, see \cite{Beauville},
 \cite{Hartshorne} or \cite{AdvancedSilverman}, III \S7, 8.

 In \refsubsubsec{subsubsec:Intersection theory}
 and \refsubsubsec{subsubsec:minimal models},
 a curve or a surface always means
 a smooth projective irreducible curve or surface
 unless otherwise mentioned.

 Let $S$ be a surface
 over the field of complex numbers $\bC$.
% over an algebraically closed field $k$ with
% characteristic zero.
 Let $\Div(S)$ be the group of all the divisors
 of $S$ and
 $\div: \bC(S)^{\times} \ra \Div(S)$
 the divisor function.
 We say that two divisors $D,D' \in \Div(S)$
 are {\it linearly equivalent}
 if $D' = D + \div(f)$ for some $f \in \bC(S)^{\times}$.
 Then there is a unique symmetric bilinear pairing
$$
 (\;\;,\;\;) : \Div(S)\times \Div(S) \ra \bZ
$$
 which satisfies the following two properties:
\begin{enumerate}
\item For curves $C_1,C_2$ on $S$
 which meet everywhere transversally,
 then $(C_1,C_2)=\# (C_1\cap C_2)$,
 where $\# (C_1\cap C_2)$ is the number of points
 of $C_1 \cap C_2$.

\item If $D,D_1,D_2 \in \Div(S)$ are
% divisors and $D_1 \sim D_2$,
 divisors and $D_1$, $D_2$ are linearly equivalent,
 then $(D, D_1) =(D, D_2)$.
 Hence the pairing $(\;\;,\;\;)$ induces
 the pairing
 $\Pic(S)\times\Pic(S)\ra \bZ$
 on the Picard group of $S$
 (the quotient group of $\Div(S)$ by the image of $\div$).
% $\div(\bC(S)^{\times})$).
\end{enumerate}
 We will also write $D\cdot D'$ in place of $(D,D')$
 for two divisors $D,D' \in \Div(S)$.
 For any divisor $D \in \Div(S)$
 we call $D^2:=D\cdot D$
 the {\it self-intersection number} of $D$.

 We say that a surface
 $S$ is a {\it fibered surface} over a curve $C$
 if there is a surjective morphism
 $\pi: S \ra C$.
 The fiber $\pi^{-1}(t)$ for a point $t \in C$
 is a smooth projective curve
 for all but finitely many points of $C$.
 For any curve $D$ on a fibered surface $S$
 over $C$
 its image $\pi(D)$ is either a point or $C$.
 A curve $D$ on $S$ is called {\it fibral}
 (or {\it vertical}) if the image is a point,
 and is called {\it horizontal}
 if the image is the curve $C$.
 A divisor $D \in \Div(S)$ is called vertical (horizontal)
 if all the curves appeared as the components of $D$
 are vertical (horizontal).
 The structure morphism $\pi:S \ra C$
 of a fibered surface induces the homomorphism
$$
 \pi^{\ast}: \Div(C)\ra \Div(S)
$$
 defined by
$$
 \sum_{t \in C} n_t [t] \mapsto \sum_{t \in C}n_t \sum_{\Gamma \subset \pi^{-1}(t)} \ord_{\Gamma}(u_t \circ \pi)[\Gamma],
$$
 where $u_t$ is a uniformizer
 of the function field $k(C)$ at $t$
 and $\Gamma$ runs through all the curves in $\pi^{-1}(t)$,
 and $\ord_{\Gamma}(u_t \circ \pi)$ is the order
 of $u_t \circ \pi$ in the function field $k(S)$ of $S$
 by the discrete valuation $\ord_{\Gamma}$ defined by $\Gamma$.
 (For the definition of the inverse image of a divisor
 in general, see, e.g. \cite{Beauville}, I or \cite{Hartshorne}, II.6.)

\begin{proposition}[\Cf \cite{Beauville}, Prop. I. 8]
\label{prop:Prop. I. 8}
\begin{enumerate}
\item
% Let $C$ be a curve
% and $\phi:S \ra C$ a surjective morphism.
 Let $\pi:S \ra C$ be a fibered surface over a curve $C$.
 If $F$ is a fiber of $\phi$
 (that is, $F = \pi^{\ast}[t]$ for some $t \in C$),
 then $F^2 = 0$.

\item
 Let $S,S'$ be surfaces
 and $g: S'\ra S$ a generically finite morphism
%%
%% namely, the fiber of
%% the generic point $g^{-1}(\eta)$ is a finite set
%%
 of degree $d$.
 If $D,D'$ are divisors on $S$,
 then $g^{\ast}D \cdot g^{\ast}D' = d (D\cdot D')$.
\end{enumerate}
\end{proposition}

\begin{lemma}[\cite{AdvancedSilverman}, Chapter III, Lemma 8.1]
\label{lem:Chapter III, Lemma 8.1}
 Let $\pi: S \ra C$ be a fibered surface
 and $\delta \in \Div(C)$.
 If $D \in \Div(S)$ is a vertical divisor
 then $D \cdot \pi^{\ast}(\delta) = 0$.
\end{lemma}

\subsubsection{Minimal models of surfaces}
\label{subsubsec:minimal models}
 
 Here we summarize some basic facts
% blow up singularity,
% $A_n$-singularity,
% exceptional curves obtained by blowing up at $A_n$-singularities,
 on minimal models
% and some facts on
% algebraic surfaces
 we need in this note.
% Basic references on algebraic surfaces are 
% \cite{Beauville}, II.

 Let $S$ be a surface over $\bC$
 and $E$ a curve on $S$.
 A curve $E$ on $S$ is called an {\it exceptional curve}
 if it is obtained as a component of the fiber of a point
 by blowing up a (possibly singular) surface.

\begin{proposition}[\Cf \cite{Beauville}, II, 1]
 Let $S$ be a smooth surface
 and $p$ a point on $S$.
 Let $\geps:\tilde{S} \ra S$ be the blow up morphism
 of $S$ at the point $p$.
 Then $\geps^{-1}(p)$ is an irreducible curve on $S$
 isomorphic to $\bP^1$.
% (such a curve is called an {\it exceptional curve} on $S$.)
\end{proposition}
\noindent
 We say that a curve $E$ is a {\it $(-1)$-curve}
 if it is isomorphic to $\bP^1$
 and its self-intersection number is $-1$,
 that is, $E^2 = -1$.

\begin{proposition}[\Cf \cite{Beauville}, Lem. II, 2,  Prop. II, 3, (i), (ii)]
\label{prop:Prop. II, 3, (i),(ii)}
 Let $S$ be a surface and $p$ a point on $S$.
 Let $\epsilon:\tilde{S} \ra S$ be the blow up at $p$
 and $E$ the exceptional curve of $p$.

\begin{enumerate}
\item
 There is an isomorphism $\Pic S \oplus \bZ \;\isom\; \Pic \tilde{S}$
 defined by $(D,n) \mapsto \epsilon^{\ast}D + n E$.

\item
 If $D,D' \in \Div S$,
 then $\epsilon^{\ast}D\cdot\epsilon^{\ast}D' = D\cdot D'$,
 $E\cdot \epsilon^{\ast}D = 0$ and $E^2 = -1$.

\item
 If $C$ is an irreducible projective curve on $S$
 which passes through the point $p$ with multiplicity $m$,
 then $\epsilon^{\ast}C = \tilde{C} + m E$
 ($\tilde{C}$ is the strict transform of $C$, namely
 the closure of $\epsilon^{-1}(C\ssm p)$ in $\tilde{S}$).
\end{enumerate}
\end{proposition}

\begin{theorem}[Castelnuovo's contractibility Theorem, \Cf \cite{Beauville}, II, Theorem 17)]
\label{thm:Castelnuovo}
 Let $S$ be a surface over $\bC$
% and $E$ a curve on $S$ which is isomorphic to $\bP^1$
% with $E^2=-1$.
 and $E$ a $(-1)$-curve on $S$.
 Then $E$ is an exceptional curve on $S$,
 namely there is a surface $S'$ and a morphism
 $\epsilon:S \ra S'$ such that $\epsilon$
 is the blow up of $S'$ at a point $p$
 with $E = \epsilon^{-1}(p)$.
\end{theorem}

 A surface $S$ is called {\it relatively minimal}
 if there is no $(-1)$-curves on $S$.
 There are only finitely many $(-1)$-curves on
 a surface.
 Therefore, for a surface $S$
 we can find a sequence of surfaces
$$
 S \ra S_1 \ra S_2 \ra \cdots \ra S_n
$$
 such that $S_n$ is relatively minimal.
 Such a surface $S_n$ is called a
 {\it relatively minimal model}
 (usually called a minimal model) of $S$.
 Note that a minimal model
 is not necessarily unique for a given surface.

 A {\it rational surface} is an irreducible smooth
 projective surface which is birational to $\bP^2$.
 It is known that,
 for any surface $S$ except rational surfaces,
 there is a unique relatively minimal model
 (the {\it minimal model} of $S$).
 For rational surfaces, the classification of relatively
 minimal models is known.
 Namely, there are two types of minimal
 models: the projective plane $\bP^2$ and 
 the Hirzebruch surfaces $\bF_n$ ($n\ge 0$, $n \neq 1$)
 (\Cf \cite{Beauville}, Theorem V.10).
 Here a {\it Hirzebruch surface} $\bF_n$ is
% a rank $2$ vector bundle $\bP_{\bP^1}(\cO_{\bP^1}\oplus\cO_{\bP^1}(-n))$
 a $\bP^1$-bundle over $\bP^1$ associated with the sheaf
 $\cO_{\bP^1}\oplus\cO_{\bP^1}(-n)$
 for $n\ge 0$,
 where $\cO_{\bP^1}$ is the structure sheaf of $\bP^1$
 and $\cO_{\bP^1}(-n)$ is the inverse of
 the $n$ tensor product of
 the Serre twisting sheaf $\cO_{\bP^1}(1)$.

 A {\it geometrically ruled surface} $S$ over
 a curve $C$ is an irreducible smooth
 projective surface together with a smooth morphism
 $\pi: S \ra C$ such that all the fibers are
 isomorphic to $C$.
 Especially geometrically ruled surfaces over $\bP^1$
 are only Hirzebruch surfaces $\bF_n$
 (see \cite{Beauville}, IV).
 The surfaces $\bF_n$ are relatively minimal for any $n \ge 0$ except $1$,
 and $\bF_1$ is isomorphic to $\bP^2$ blown up at one point.

\subsection{Proof of \reflem{lem:geom.ruled}}
\label{subsec:proof of first lem}

 To show \reflem{lem:geom.ruled},
 here we compute the surface $T_i$ obtained by
 blowing down all the $(-1)$-curves on $\tilde{S}_i$
 appeared as the components of fibers of the morphism
 $\tilde{S}_i \ra S_i \xra{\phi_i} \bP^1$.
 Then we see that $T_i$ is a geometrically ruled surface
 over $\bP^1$.
 In \refsubsec{subsec:proof of second lem}
 we show that $T_i$ is a minimal model of $\tilde{S}_i$
 for each $i$.

%%%%%%%%%%%%%%%%%%%%%%
%%%%%%%%%%%% Whitehead link minimal model
%%%%%%%%%%%%%%%%%%%%%

\subsubsection{Whitehead link case}
%\subsection{Whitehead link case}

 As we have seen in \refsec{sec:ProjModel},
 $\phi_0:S_0 \ra \bP^1$ has six degenerate fibers.
 Thus, the composite morphism of $\phi_0$ and
 the blow-up morphism
 $\tilde{\phi}_0:\tilde{S}_0 \ra S_0 \xra{\phi_0} \bP^1$
 also has six degenerate fibers.
 Since the other fibers are conics in $\bP^2$
 they are isomorphic to $\bP^1$.
 Here we show that
 we can blow down $\tilde{S}_0$ over $\bP^1$
 some number of times
 so that it becomes
 a geometrically ruled surface over $\bP^1$.
%
% The fiber
% $\phi^{-1}_2(0:1)$ consists of two $\bP^1$ which intersects
% each other transversally.
% There is no singular points of $S_2$ on this fiber.
% Hence it is same on $\tilde{S}_2$.
% We can easily see that the divisor $\phi^{\ast}_2(0:1)$
% is $E'_1+E'_2$, where $E'_i$ is isomorphic to $\bP^1$
% with $E'^2_i=-1$.
% Hence we can blow down this fiber.

 First consider the $\phi^{-1}_0(0:1)$ case.
 The fiber $\phi^{-1}_0(0:1)$
 consists of two $\bP^1$ which intersects
 each other transversally.
 There is no singular points of $S_0$ on this fiber.
% Hence they have isomorphic inverse image in $\tilde{S}_0$.
 Hence it is isomorphic to
 $\tilde{\phi}^{-1}_0(0:1)$.
% Then we see that the divisor $\tilde{\phi}^{\ast}_0(0:1)$
% is written as $C_1+C_2$, where $C_i$ is isomorphic to
% $\bP^1$ with $C^2_i=-1$.
% Hence we can blow down this fiber
% by the Castelnuovo's contractibility theorem.
%
%
% First note that, the four fibers
% $\phi^{-1}_1(1:\tfrac{\sqrt{5}\pm 1}{2})$ and
% $\phi^{-1}_1(1:-\tfrac{\sqrt{5}\pm 1}{2})$
% are the unions of two $\bP^1$s intersecting each other
% transversally at one point,
% and they do not contain any singular point of the surface $S_1$.
% Hence they are isomorphic to
% $\tilde{\phi}^{-1}_1(1:\tfrac{\sqrt{5}\pm 1}{2})$ and
% $\tilde{\phi}^{-1}_1(1:-\tfrac{\sqrt{5}\pm 1}{2})$.
%
%
 Now we show that
 each curve $C_i$ in the fiber
 has self-intersection number $-1$.
% (\Cf \cite{Beauville}, Prop. I. 8 or
%  Prop. II, 3,
% \cite{AdvancedSilverman}, Chapter III, Lemma 8.1).
% Let $F = C_1 + C_2$ be the divisor defined by
% one of the four fibers,
% where $C_i$ are curves constituting the fiber.
%
\noindent
 By \reflem{lem:Chapter III, Lemma 8.1},
 $(C_1+C_2)\cdot C_i = 0$.
 Since $(C_1+C_2)\cdot C_i = C^2_i + 1$,
 we have $C^2_i = -1$.
 Thus we can blow down one of these two $(-1)$-curves
 in the fiber
 by \refthm{thm:Castelnuovo},
% Castelnuovo's contractibility Theorem
% (\Cf \cite{Beauville}, II, Theorem 17)
 and the fiber becomes a curve $C \;\!\isom\;\!\bP^1$
 with self-intersection number $0$
 by \refprop{prop:Prop. II, 3, (i),(ii)} $(2),(3)$.
\begin{figure}[h]
\begin{center}
%\centering
\quad\quad
\xymatrix{
 \ar@{-}[ddrr] & & C_1  &         &            &  & \\
               & &      & \ar[r]^{\epsilon}  & \ar@{-}[rr] & & C\\
 \ar@{-}[uurr] & & C_2  &         &            &  &
}
\end{center}
\end{figure}
 We can work on the two fibers $\phi^{-1}_0(1:\pm\tfrac{1}{\sqrt{2}})$
 completely in the same way.

%
% The fibers $\phi^{-1}_0(1:0)$ and $\phi^{-1}_0(1:\pm 1)$
% have the same shape
% as $\phi^{-1}_1(1:0)$ and $\phi^{-1}_1(1:\pm 1)$
% which appeared in the $6^2_2$ case.
%

 Next we consider the case
 $\phi^{-1}_0(1:0)$.
% (the same argument also works for $\phi^{-1}_1(0:1)$).
 As we have seen in
 \refsubsec{subsec:Whitehead sec:ProjModel},
 the fiber $\phi^{-1}_0(1:0)$ is a double $\bP^1$ line $C$
 which contains two $A_1$ singular points of $S_0$.
 Then the divisor $\tilde{\phi}^{\ast}_0(1:0)$
 in $\Div \tilde{S}_0$ is written as
 $2C' + E_1 + E_2$,
 where $\tilde{\phi}_0:\tilde{S}_0\ra S_0 \xra{\phi_0} \bP^1$,
 the curve $C'$ is the strict transform of $C$ in $\tilde{S}_0$
 and $E_i$ are the exceptional curves which are isomorphic to $\bP^1$
 with $E_i^2=-2$ and $E_i\cdot C'=1$.
 Since $\tilde{\phi}^{\ast}_0(1:0)\cdot C'=0$,
 we have $C'^2=-1$.
 Hence we can blow down at $C'$ and
 the fiber becomes $E'_1 + E'_2$ with $E'^2 = -1$
 by \refprop{prop:Prop. II, 3, (i),(ii)} $(2),(3)$.
 Thus we can blow down again and
 the fiber becomes $\bP^1$.

\begin{figure}[h]
\begin{center}
%\centering
\quad\quad
\xymatrix{
 & E_1\ar@{-}[dd]  & E_2\ar@{-}[dd]&         & E'_1\ar@{-}[ddrr] & & & &            & & & \\
 C'\ar@{=}[rrr] & &     & \ar[r]^{\epsilon'}  & & & \ar[r]^{\epsilon}&  \ar@{-}[rr] & & \\
              & &           &                & E'_2\ar@{-}[uurr] & & & &          & & &
}
\end{center}
\end{figure}

 Finally we consider the case
% $\phi^{-1}_0(1:\pm 1)$.
 $\phi^{-1}_0(1: 1)$.
 We only have to consider the case $\phi^{-1}_0(1: 1)$
 (we can work on $\phi^{-1}_0(1: -1)$ similarly).
 It consists of two rational curves $C_1$ and $C_2$ which
% are isomorphic to $\bP^1$
 intersects each other
 transversally at one point $(1:1:0,\;1:1)$,
 which is an $A_1$ singular point of $S_0$.
 Hence the divisor $\tilde{\phi}^{\ast}_0(1:1)$ is $C'_1+C'_2+E$,
 where $C'_i$ is the strict transform of $C_i$ on $\tilde{S}_0$
 and $E$ is the exceptional curve obtained by resolving
 the singularity at $(1:1:0,\;1:1)$.
\begin{figure}[h]
\begin{center}
%\centering
\quad\quad
\xymatrix{
 & C'_1\ar@{-}[dd]  & C'_2\ar@{-}[dd]&         & C'_2\ar@{-}[ddrr] & & & &            & & & \\
 E\ar@{-}[rrr] & &     & \ar[r]^{\epsilon'}  & & & \ar[r]^{\epsilon}&  \ar@{-}[rr] & &E'' \\
              & &           &                & E'\ar@{-}[uurr] & & & &          & & &
}
\end{center}
\end{figure}
\noindent
 The exceptional curve $E$ intersects
 with $C'_1$, $C'_2$ transversally
 at one point respectively
 and $C'_1$ and $C'_2$ do not meet at any point.
 Hence $E\cdot C'_i=1$ and $C'_1 \cdot C'_2=0$.
 Since $E^2=-2$,
 we have $C'^2_i=-1$.
 Thus we can blow down at $C'_1$ and $C'_2$
 and obtain one rational curve $E''$ with self-intersection number $0$.

 Therefore the surface $T_0$ obtained by
 blowing down
 all the degenerate fibers of $\tilde{S}_0$ over $\bP^1$
 is a geometrically
 ruled surface over $\bP^1$.

%%%%%%%%%%%%%%%%%%%%%%%%%%%
%%%%%%%%%%%%%% 6^2_2 minimal model
%%%%%%%%%%%%%%%%%%%%%%%%%%

\subsubsection{$6^2_2$ case}
%\subsection{$6^2_2$ case}

 As we have seen in \refsec{sec:ProjModel},
 $\phi_1:S_1 \ra \bP^1$ has eight degenerate fibers.
 Thus the composite morphism
 $\tilde{\phi}_1:\tilde{S}_1 \ra S_1 \xra{\phi_1} \bP^1$ also has eight degenerate fibers.
 Here we show that the surface $T_1$ obtained by
 blowing down all the degenerate fibers of $\tilde{S}_1$
 is also a geometrically ruled surface.

 First note that, the four fibers
 $\phi^{-1}_1(1:\tfrac{\sqrt{5}\pm 1}{2})$ and
 $\phi^{-1}_1(1:-\tfrac{\sqrt{5}\pm 1}{2})$
 are the unions of two $\bP^1$s intersecting each other
 transversally at one point,
 and they do not contain any singular point of the surface $S_1$.
% Hence they are isomorphic to
% $\tilde{\phi}^{-1}_1(1:\tfrac{\sqrt{5}\pm 1}{2})$ and
% $\tilde{\phi}^{-1}_1(1:-\tfrac{\sqrt{5}\pm 1}{2})$.
%
 Hence the situation is the same as the $\phi^{-1}_0(0:1)$ case.
 By the same manner as in the $\phi^{-1}_0(0:1)$ case,
 we can show that
 each curve in the four fibers have self-intersection number
 $-1$ and can be blown down
 to obtain one $\bP^1$.

 The fibers $\phi^{-1}_1(1:0)$ and $\phi^{-1}_1(1:\pm 1)$
 have the same shape and singular points
 as $\phi^{-1}_0(1:0)$ and $\phi^{-1}_0(1:\pm 1)$
 which appeared in the Whitehead link case.
 Therefore we see that
 the fibers $\tilde{\phi}^{-1}_1(1:0)$ and
 $\tilde{\phi}^{-1}_1(1:\pm 1)$ can be blown down
 and we have $\bP^1$.
 Note that
 the same argument as $\phi^{-1}_1(1:0)$
 also works for $\phi^{-1}_1(0:1)$
 since they are symmetric.
%
% Next we consider the case
% $\phi^{-1}_1(1:0)$
% (the same argument also works for $\phi^{-1}_1(0:1)$).
%

 Hence the surface $T_1$ obtained by
 blowing down $\tilde{S}_1$ over $\bP^1$
 is a geometrically
 ruled surface over $\bP^1$.

\subsubsection{$6^2_3$ case}
%\subsection{$6^2_3$ case}

% We can obtain a minimal model of $\tilde{S}_2$
% by the similar argument as in the $6^2_2$ case.
 As we have seen in the previous section,
 $\phi_2:S_2 \ra \bP^1$ has four degenerate fibers.
 We can work on the fiber
 $\phi^{-1}_2(1:0)$ (resp. $\phi^{-1}_2(0:1)$)
 in the same way
 as in the $\phi^{-1}_0(1:0)$ (resp. $\phi^{-1}_0(0:1)$) case.

 Now we consider the $\phi^{-1}_2(1: 1)$
 (resp. $\phi^{-1}_2(1:-1)$) case.
 Since the point $(1:1:0,\;1:1)$ (resp. $(1:-1:0,\;1:-1)$)
 is an $A_3$ singular point,
 the exceptional curve consists of
 three rational curves $E_1$, $E_2$ and $E_3$
 with $E^2_i=-2$,
 $E_1\cdot E_2= E_3\cdot E_2=1$ and $E_1\cdot E_3=0$.
 Then the divisor $\tilde{\phi}^{\ast}_2(1:1)$
 (resp. $\tilde{\phi}^{\ast}_2(1:-1)$)
 is $2C'+E_1+2E_2+E_3$
 (see
 \reffig{fig:Blow down of (1:pm 1)}),
 where
 $C'$ is the strict transform on $\tilde{S}_2$
% of the curve on $S_2$ defined by $x\mp y=0$.
 of the curve on $S_2$ defined by $x- y=0$ (resp. $x+y=0$).

 Since $\tilde{\phi}^{\ast}_2(1:1)^2 = 0$
 (resp. $\tilde{\phi}^{\ast}_2(1:\pm 1)^2 = 0$),
 we have $C'^2=-1$.
 Hence we can also blow down this fiber for $C'$,
 and the divisor of fiber on the blow down is
 $E_1 + 2 E'_2 + E_3$ with $E'^2_2 = -1$ and
 $E_1\cdot E'_2 = E_3\cdot E'_2 = 1$.
 Now the situation is the same as
 the $\tilde{\phi}^{\ast}_0(1:0)$ case.
 We can blow down the fiber twice to have one $\bP^1$.
 Thus we obtain a geometrically ruled surface $T_2$
 over $\bP^1$ by blowing down $\tilde{S}_2$ repeatedly.

%
%\begin{figure}[h]
%\begin{center}
%%\centering
%\quad\quad
%\xymatrix{
% & E_1\ar@{-}[dd] & C'\ar@{=}[dd] & E_3\ar@{-}[dd]&         & E'_1\ar@{-}[ddrr] & & & &            & & & \\
% E_2\ar@{=}[rrrr] & & &     & \ar[r]^{\epsilon''}  & & & \ar[r]^{\epsilon}&  \ar@{-}[rr] & & \\
%              & & &          &                & E'_3\ar@{-}[uurr] & & & &          & & &
%}
%\end{center}
%\end{figure}
%

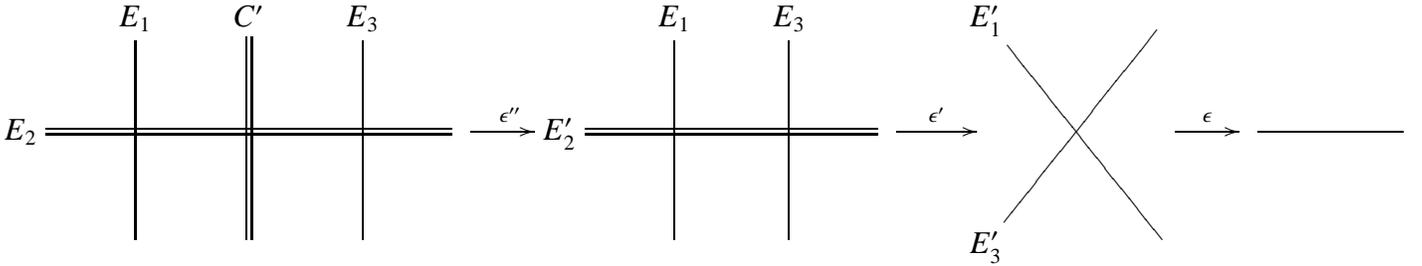
\begin{figure}[h]
\centerline{
\hspace{20mm}
%\begin{center}
%\centering
\xymatrix{
 & E_1\ar@{-}[dd] & C'\ar@{=}[dd] & E_3\ar@{-}[dd]&         &
 & E_1\ar@{-}[dd] & E_3\ar@{-}[dd] & &
 E'_1\ar@{-}[ddrr] & & & &            & & & \\
 E_2\ar@{=}[rrrr] & & &     & \ar[r]^{\epsilon''}  & 
 E'_2\ar@{=}[rrr] & & & \ar[r]^{\epsilon'} & 
& & \ar[r]^{\epsilon}&  \ar@{-}[rr] & & \\
              & & &          &                &
 & & & &
 E'_3\ar@{-}[uurr] & & & &          & & &
}
}
%\end{center}
\caption{Blow down of $\tilde{\phi}^{-1}_2(1:\pm 1)$}
\label{fig:Blow down of (1:pm 1)}
\end{figure}

\subsection{Proof of \reflem{lem:determ.ofT_i}}
\label{subsec:proof of second lem}
%
% The determination of the Hirzebruch surfaces $T_i$
% is done by the completely same argument,
% though it depends on the explicit defining equations
% of $S_i$.
%

 In \refsubsec{subsec:proof of first lem}
 we have shown that the surfaces $T_i$
 obtained by blowing down the degenerate fibers
 of $\tilde{S}_i$ are geometrically ruled surfaces.
 Since Hirzebruch surfaces are the only rational surfaces
 which are geometrically ruled surfaces,
 each $T_i$ is isomorphic to
 a Hirzebruch surface $\bF_n$ for some $n\ge 0$.
 It remains to determine the number $n$.
% Here we show that $T_i$ already are minimal models
% of $\tilde{S}_i$.
 In fact we show that $T_i$ are
 the Hirzebruch surface $\bF_0 = \bP^1\times\bP^1$.
 Therefore $T_i$ are minimal models of $\tilde{S}_i$.
 Here we review a proposition on the Hirzebruch surfaces.

\begin{proposition}[\Cf \cite{Beauville}, Prop. IV.1]
\label{prop:Hirzebruch surface}
 Let $\bF_n$ be a Hirzebruch surface ($n \ge 0$).
 Let $f \in \Pic \bF_n$ be the element defined by a fiber of $\bF_n$ over $\bP^1$
 and let $h \in \Pic \bF_n$ be the element corresponding
 to the tautological line bundle
 (that is, the invertible sheaf $\cO_{\bF_n}(-1)$).
\begin{enumerate}
\item $\Pic \bF_n = \bZ h \oplus \bZ f$ with $f^2=0$ and $h^2 = n$.
\item When $n > 0$,
 there exists a unique irreducible projective curve $B$ on $\bF_n$
 with $b = [B]^2 < 0$.
% If $b = [B] \in \Pic\bF_n$,
 $b$ is written as $b = h - nf$.
 Thus $b^2 = - n$.
\item If $n \neq m$,
 two surfaces $\bF_n$ and $\bF_m$ are not isomorphic.
 $\bF_n$ is a minimal model except if $n = 1$.
 $\bF_0$ is $\bP^1\times\bP^1$ and
 $\bF_1$ is isomorphic to $\bP^2$ with one point blown up.
\end{enumerate}
\end{proposition}

\noindent
 Note that we can take two sections $s_k:\bP^1 \ra S_i$
 which are defined by
\begin{align*}
 s_1&:(z:w) \mapsto (z:w:0,\; z:w),\\
 s_2&:(z:w) \mapsto (w:z:0,\; z:w).
\end{align*}
\noindent
 These two sections meet on $S_i$ at
% two points $(1:\pm 1:0,\; 1:\pm 1)$.
 the two points $(1:1:0,\; 1:1)$ and $(1:-1:0,\; 1:-1)$.
 We see that their lifts $\tilde{s}_k$ on
 the desingularization $\tilde{S}_i$
 (that is, the strict transforms of $s_k(\bP^1)$)
 do not intersect on the exceptional curves at
 $(1:1:0,\; 1:1)$ and $(1:- 1:0,\; 1:- 1)$,
 which means they do not intersect on $\tilde{S}_i$.
 Considering the process of blowing downs
 in \refsubsec{subsec:proof of first lem}
 (see the figures of $\phi^{-1}_i(1: \pm 1)$),
 we know that their images in $T_i$ also do not intersect
 each other.
% Since these sections are linearly equivalent,
 Note that $s_1(\bP^1) = V(xw-yz,u)$,
 $s_2(\bP^1) = V(xz-yw,u)$
 and
\begin{align*}
 F_0 &= u^2z(z^2-2w^2) + w(xz-yw)(xw-yz), \\
 F_1 &= u^2(z^4-3z^2w^2+w^4) + zw(xz-yw)(xw-yz), \\
 F_2 &= u^2z(z^2-w^2) + w(xz-yw)(xw-yz).
\end{align*}
 Then we can check that
 in $\Div(\tilde{S}_i)$
$$
 \div(xw-yz) = [V(xw-yz,u)], \quad
 \div(xz-yw) = [V(xz-yw,u)].
$$
 Thus we see that
 $\tilde{s}_k(\bP^1)$ are linearly equivalent.
 Therefore we have
$$
% \deg(N_{S_1/s_1(\bP^1)}=
\tilde{s}_1(\bP^1)^2 = \tilde{s}_2(\bP^1)^2 = \tilde{s}_1(\bP^1)\cdot\tilde{s}_2(\bP^1)=0.
$$
\noindent
 The same is true for the images of $\tilde{s}_k(\bP^1)$
 in $T_i$.
% where $N_{S_1/s_1(\bP^1)}$ is the normal bundle of $s_1(\bP^1)$
% on $S_1$.
\begin{lemma}
 Let $C$ be an irreducible projective curve on
 a Hirzebruch surface $\bF_n$
 which is different from $B$.
 If $n>0$,
 then $[C]^2= 0$ if and only if
% $[C] \in \bZ f \subset \Pic\bF_n$.
 $C$ is a fiber of $\bF_n$.
\end{lemma}
\begin{proof}
 We use the notation in \refprop{prop:Hirzebruch surface}.
 Put $b = [B]$ and $c = [C]$.
 Since $\Pic \bF_n = \bZ h \oplus \bZ f$,
 there exist $\ga,\gb \in \bZ$ such that
 $c = \ga h + \gb f$.
 First note that $c\cdot f \ge 0$ and $c \cdot b \ge 0$
 since $B \neq C$.
 By \refprop{prop:Hirzebruch surface},
 we know that $f^2 = 0$, $h\cdot f = 1$ and $b = h - n f$.
%
% (a h + b f)(h - n f) = an - an + b - b n 0 = b
% (a h + b f)(a h + b f) = a^2n + 2 a b
%
 Hence $\ga \ge 0$ and $\gb \ge 0$.
 Since $c^2 = \ga^2n + 2 \ga\gb$,
 we see that $c^2 = 0$
 if and only if $\ga = 0$.
% Therefore $C$ is linearly equivalent to the fiber
% represented by $f$.
 If $C$ is not a fiber on $\bF_n$,
 then the restriction morphism $C \hr \bF_n \ra \bP^1$
 is surjective.
 This means $c\cdot f > 0$.
 Therefore $C$ is a fiber of $\bF_n$.
\end{proof}
\noindent
%
%
% Note that a Hirzebruch surface $\bF_n$ ($n>0$) only has
% irreducible curves with self-intersection number $0$
% as fibers over $\bP^1$
% (\Cf the proof of (ii) in \cite{Beauville}, Prop. IV.1).
 Remember that the above sections $\tilde{s}_k$ are global sections.
% Hence we conclude that the minimal surface is $\bF_0=\bP^1\times\bP^1$.
 Hence we conclude that $n=0$,
 namely the geometrically ruled surface $T_i$
 over $\bP^1$ is $\bF_0=\bP^1\times\bP^1$,
 which is already a minimal model of $\tilde{S}_i$.
% It is known (\Cf \cite{Beauville}, Prop. IV.1) that
% the Picard group $\Pic(S/\bP^1)$ is generated by two elements,
% one comes from an arbitrary fiber of $\phi_1:S_1\ra \bP^1$
% and the other is a section $s:S_1\ra \bP^1$ whose corresponding
% quotient line bundle is isomorphic to $\cO_{\bP^1}$.
% Thus we see that the minimal surface is equal to $\bF_0=\bP^1\times\bP^1$.

%
% For the determination of the minimal surface
% we can use the same argument for the $6^2_2$ case.
% Namely,
% there are two sections $s_i:\bP^1 \ra S_0$
% which are defined by
%\begin{align*}
% s_1&:(z:w) \mapsto (z:w:0,\; z:w),\\
% s_2&:(z:w) \mapsto (w:z:0,\; z:w).
%\end{align*}
%\noindent
% These two sections meets on $S_0$ at $(1:\pm 1:0,\; 1:\pm 1)$.
% Then it is easy to see that 
% their lifts $\tilde{s}_i$ on
% the desingularization $\tilde{S}_0$ which do not intersect.
%% but they are linearly equivalent.
%%F:=F_0:=u^2z^3 -xyz^2w + (x^2+y^2-2u^2)zw^2 -xyw^3=g + h u^2
%% g=-xyz^2w+(x^2+y^2)zw^2-xyw^3=-w(xz-yw)(xw-yz), \quad
%% h=z(z^2-2w^2).
%%
%% In S=V(F), w(xz-yw)(xw-yz)=u^2 h
%% div(w)+div(xz-yw) +div(xw-yz)=2div(u)+div(h)
%% div(xz-yw)=div(xw-yz) + r
%% r=u^2 h/w(xw-yz)^2
%%
% Hence, their self-intersection numbers are zero.
% Thus the minimal surface of $\tilde{S}_0$ is
% $\bP^1\times\bP^1$.
%

\begin{remark}
\label{rem:idealcurve}
 The set $S_i \ssm V(p_i)$ of ideal points of
 the canonical component in $\bP^2\times\bP^1$ consists of
 three `ideal curves' for $i=0, 2$, that is,
 one fiber $\phi^{-1}_i(1:0)$
 and two global sections $s_k(\bP^1)$.
 For $S_1$, there is an additional fiber.
 Namely $S_1 \ssm V(p_1)$ consists of $\phi^{-1}_1(1:0)$,
 $\phi^{-1}_1(0:1)$ and two global sections $s_k(\bP^1)$.
% We see from the argument done in this subsection,
% these two horizontal curves define a same divisor
% in the Picard group $\Pic \bF_0$ of the minimal model of $S_i$.
% The vertical curve above contracts one point.
% One of the exceptional curve remains as the fiber at $(1:0)$.
% Then these two divisors are the basis of $\Pic \bF_0\;\isom\; \bZ\oplus\bZ$ (\Cf \cite{Beauville}, Prop. IV. 1).
\end{remark}

%%%%%%%%%%%%%%%%%%%%%%%%%%%%%%%%%%%%%%%%%%%%
%%%%%%%%%%%% subsection %%%%%%%%%%%%%%%%%%5
%%%%%%%%%%%%%%%%%%%%%%%%%%%%%%%%%%%%%%%%%%%

%\subsection{Desingularization of the projective models of canonical components}
%\label{subsec:Desingularization}
\section{Desingularization of the models}
\label{sec:Desingularization}

 In this section we determine $\tilde{S}_i$
 in terms of the number of blow ups from their minimal models
 we have computed in \refsec{sec:MinModel}.

 From the result in \refsec{sec:MinModel},
% By \refprop{Prop1Landes}
 the smooth surfaces $\tilde{S}_i$ are
 isomorphic to $\bP^1\times\bP^1$ with one point blown up
% $n$ times for some non-negative integer $n$.
 some number of times.
 Suppose that $\tilde{S}_i$ is obtained from $\bP^1\times\bP^1$
 by $n$ one-point blow ups.
 Since $\chi(\bP^1\times\bP^1) = 4$
 we have
 $\chi(\tilde{S}_i) = \chi(\bP^1\times\bP^1) + n = n + 4$
 by repeatedly using \reflem{lem:EulerCharSmoothPt}.
 To determine the number $n$ we have to compute
 $\chi(\tilde{S}_i)$.
 This is done by comparing the Euler characteristics
 of $S_i$ and $\tilde{S}_i$.
% employing the results obtained in \refsec{sec:MinModel}.
 For the computation of $\chi(S_i)$
% (hence $\chi(\tilde{S}_i)$),
 we follow the Landes' method in \cite{EmilyLandes}, \S$4$.
 Here we introduce a rational map
 $\vphi : \bP^2 \times\bP^1 \dra \bP^1 \times \bP^1$
 defined by
 $(x:y:u,\;z:w) \mapsto (x:y,\;z:w)$
 and set $\vphi_i:= \vphi\!\!\mid_{S_i}$.
 This plays a crucial role for the computation of $\chi(S_i)$
 in this section.

\subsection{Whitehead link case}

 The computation of $\chi(S_0)$ has already been done
 in \cite{EmilyLandes}, \S$4$.
 Here we review the computation for the completeness
 of this note.
 For more details,
 see \opcit  or the $6^2_2$ case.
% We can compute $\chi(S_0)$ by the same way.
% Hence we omit the details.

 Let $\vphi_0 : S_0 \hr \bP^2 \times \bP^1 \dra \bP^1 \times \bP^1$
 be the rational map defined by $(x:y:u,\;z:w) \mapsto (x:y,\;z:w)$.
 This is not defined at the three points
 $(0:0:1,\;0:1)$ and $(0:0:1,\;1:\pm \frac{1}{\sqrt{2}})$.
 Let $P_0$ be the set of those three points and put $U_0:=S_0 \ssm P_0$.
 The image $\Im(\vphi_0)$ of $U_0$ is
 $\bP^1 \times \bP^1 \ssm Q_0$,
 where
\begin{align*}
  Q_0
   &= \bP^1 \times \{(0:1)\} \ssm \left\{(1:0,\;0:1),\; (0:1,\;0:1)\right\} \\
     &\sqcup  \bP^1 \times \left\{\left(1:\tfrac{1}{\sqrt{2}}\right)\right\} \ssm \left\{\left(1:\sqrt{2},\;1:\tfrac{1}{\sqrt{2}}\right),\;\left(1:\tfrac{1}{\sqrt{2}},\;1:\tfrac{1}{\sqrt{2}}\right)\right\} \\
     &\sqcup  \bP^1 \times \left\{\left(1:-\tfrac{1}{\sqrt{2}}\right)\right\} \ssm \left\{\left(1:-\sqrt{2},\;1:-\tfrac{1}{\sqrt{2}}\right),\;\left(1:-\tfrac{1}{\sqrt{2}},\;1:-\tfrac{1}{\sqrt{2}}\right)\right\}.
\end{align*}
\noindent
 Hence $\chi(Q_0)=0$.
 Let $L_0$ be the set of the above six points
\begin{align*}
 \left(1:0,\;0:1\right),\;\left(0:1,\;0:1\right),\;\left(1:\sqrt{2},\;1:\tfrac{1}{\sqrt{2}}\right),\\
\left(1:\tfrac{1}{\sqrt{2}},\;1:\tfrac{1}{\sqrt{2}}\right),\;\left(1:-\sqrt{2},\;1:\tfrac{-1}{\sqrt{2}}\right),\;\left(1:\tfrac{-1}{\sqrt{2}},\;1:\tfrac{-1}{\sqrt{2}}\right).
\end{align*}
%
%\begin{align*}
% \text{Fiber of }\left\{\left(1:\sqrt{2},\;0:1\right)\right\} \text{ in }U
%&= \left\{\left(1:\sqrt{2}: u,\;0:1\right)\right\} \isom \bA^1, \\
% \text{Fiber of }\left\{\left(1:\tfrac{1}{\sqrt{2}},\;0:1\right)\right\} \text{ in }U
%&= \left\{\left(1:\tfrac{1}{\sqrt{2}}: u,\;0:1\right)\right\} \isom \bA^1, \\
% \text{Fiber of }\left\{\left(1:\sqrt{2},\;1:\tfrac{1}{\sqrt{2}}\right)\right\} \text{ in }U
%&= \left\{\left(1:\sqrt{2}: u,\;1:\tfrac{1}{\sqrt{2}}\right)\right\} \isom \bA^1, \\
% \text{Fiber of }\left\{\left(1:\tfrac{1}{\sqrt{2}},\;1:\tfrac{1}{\sqrt{2}}\right)\right\} \text{ in }U
%&= \left\{\left(1:\tfrac{1}{\sqrt{2}}: u,\;1:\tfrac{1}{\sqrt{2}}\right)\right\} \isom \bA^1, \\
% \text{Fiber of }\left\{\left(1:-\sqrt{2},\;1:\tfrac{-1}{\sqrt{2}}\right)\right\} \text{ in }U
%&= \left\{\left(1:-\sqrt{2}: u,\;1:\tfrac{-1}{\sqrt{2}}\right)\right\} \isom \bA^1, \\
% \text{Fiber of }\left\{\left(1:\tfrac{-1}{\sqrt{2}},\;1:\tfrac{-1}{\sqrt{2}}\right)\right\} \text{ in }U
%&= \left\{\left(1:\tfrac{-1}{\sqrt{2}}: u,\;1:\tfrac{-1}{\sqrt{2}}\right)\right\} \isom \bA^1, 
%\end{align*}
%
\noindent 
 Let
$$
 F_0=u^2z^3 -xyz^2w + (x^2+y^2-2u^2)zw^2 -xyw^3=G_0 + H_0 u^2
$$
 be the decomposition of $F_0$
 in terms of the power of $u$,
 where
$$
 G_0=-xyz^2w+(x^2+y^2)zw^2-xyw^3, \quad
 H_0=z(z^2-2w^2).
$$
 For $(z:w) \in \bP^1$, we see that
 $H_0(z,w)=0$ if and only if
 $(z:w)=(0:1),\; (1:\pm \frac{1}{\sqrt{2}})$.
 Then we can check that
 the set $\{G_0=H_0=0\} \subset \bP^1\times\bP^1$ is equal to $L_0$.
 Therefore each point of $L_0$ has an infinite fiber isomorphic to
 the affine line $\bA^1$.
 Hence we have $\chi(L_0)=6$ and $\chi(\vphi^{-1}_0(L_0))=6$.
%
% Let $F:=F_0:=u^2z^3 -xyz^2w + (x^2+y^2-2u^2)zw^2 -xyw^3$
% be the homogenization of the polynomial
% $P=P_0=z^3 - xyz^2 + (x^2 + y^2 -2)z -xy$
% in the projective space $\bP^2 \times \bP^1$.
 Since $G_0=w(xz-yw)(xw-yz)$,
 the set $B_0:=V(G_0) \subset \bP^1\times\bP^1$ is decomposed
 into the following three subsets
\begin{align*}
 B_{01}&=V(w)=\bP^1 \times \{(1:0)\} \subset \bP^1\times\bP^1, \\
 B_{02}&=V(xz-yw)=\{(1:y,\; y:1),\; (0:1,\; 1:0)\} \;\isom\; \bP^1, \\
 B_{03}&=V(xw-yz)=\{(1:y,\; 1:y),\; (0:1,\; 0:1)\} \;\isom\; \bP^1.
\end{align*}
 Their relations are as follows:
\begin{align*}
 B_{01} \cap B_{02} &=\{(0:1,\; 1:0)\}, \quad B_{01}\cap B_{03}=\{(1:0,\;1:0)\}, \\
 B_{02} \cap B_{03} &= \{(1:1,\; 1:1),\; (1:-1,\; -1:1)\}, \quad  B_{01} \cap B_{02} \cap B_{03} =\emptyset.
\end{align*}
\noindent
 Hence 
\begin{align*}
\chi(B_0)&=\chi(B_{01}\cup B_{02} \cup B_{03})\\
         &=\chi(B_{01}\cup B_{02})+\chi(B_{03})
          -\chi(B_{01}\cap B_{03} \;\cup\; B_{02}\cap B_{03}) \\
         &=\chi(B_{01})+\chi(B_{02})+\chi(B_{03})
         -\chi(B_{01}\cap B_{02})-\chi(B_{01}\cap B_{03})-\chi(B_{02}\cap B_{03})
         + \chi(B_{01}\cap B_{02} \cap B_{03})\\
         &=2+2+2-1-1-2+0=2.
\end{align*}
 Thus we have
\begin{align*}
 \chi(U_0) &= 2 \chi(\bP^1 \times \bP^1 \ssm (B_0 \sqcup Q_0))
            + \chi(B_0 \ssm L_0) + \chi(\vphi^{-1}_0(L_0)) \\
           &= 2\chi(\bP^1 \times \bP^1) -\chi(B_0) -2\chi(Q_0)
            - \chi(L_0) + \chi(\vphi^{-1}_0(L_0))\\
%           &= 2\times 4 -2  -5 + 5=6.\\
           &= 2\times 4 -2  -6 + 6=6.\\
 \chi(S_0) &= \chi(U_0) + \chi(P_0)=\chi(U_0) + 3=9.
\end{align*}
\noindent
 Thus
 $\chi(\tilde{S}_0)=\chi(S_0\ssm S_{0,\sing}) + 4 \chi(\bP^1)=\chi(S_0)-4+8=13$.
% The surface $\tilde{S}$ is birational to
% a conic bundle (with singular fibers).
 Therefore $\tilde{S}_0$ is isomorphic to $\bP^1\times\bP^1$
% blown up at $13-4=9$ points.
 blown up $13-4=9$ times.

%\subsubsection{$6^2_2$ case}
\subsection{$6^2_2$ case}

 Let $\vphi_1 : S_1 \hr \bP^2 \times\bP^1 \dra \bP^1 \times \bP^1$
 be the rational map defined by
 $(x:y:u,\;z:w) \mapsto (x:y,\;z:w)$.
 This is not defined at the following four points
$$
P_1:=\left\{\left(0:0:1,\;1:\pm \tfrac{1 \pm \sqrt{5}}{2}\right)\right\}.
$$
 The image $\Im(\vphi_1)$ of the open subset
 $U_1:=S_1\ssm P_1$ is $\bP^1 \times \bP^1 \ssm Q_1$,
 where
\begin{align*}
 Q_1 &= \bP^1 \times \left\{\left(1:\tfrac{1+\sqrt{5}}{2}\right)\right\} \ssm \left\{\left(1:\tfrac{\pm 1+\sqrt{5}}{2},\;1:\tfrac{1+\sqrt{5}}{2}\right)\right\} \\
     &\sqcup \bP^1 \times \left\{\left(1:\tfrac{1-\sqrt{5}}{2}\right)\right\} \ssm \left\{\left(1:\tfrac{\pm 1-\sqrt{5}}{2},\;1:\tfrac{1-\sqrt{5}}{2}\right)\right\} \\
     &\sqcup \bP^1 \times \left\{\left(1:\tfrac{-1+\sqrt{5}}{2}\right)\right\} \ssm \left\{\left(1:\tfrac{\pm 1 + \sqrt{5}}{2},\;1:\tfrac{-1+\sqrt{5}}{2}\right)\right\} \\
     &\sqcup \bP^1 \times \left\{\left(1:\tfrac{-1-\sqrt{5}}{2}\right)\right\} \ssm \left\{\left(1:\tfrac{\pm 1 - \sqrt{5}}{2},\;1:\tfrac{-1-\sqrt{5}}{2}\right)\right\}. 
\end{align*}
\noindent
 Let $L_1$ be the subset of $\bP^1 \times \bP^1$ which
 consists of the eight points
$$
\left(1:\tfrac{\pm 1 +\sqrt{5}}{2},\;1:\tfrac{1+\sqrt{5}}{2}\right),\; \left(1:\tfrac{\pm 1 -\sqrt{5}}{2},\;1:\tfrac{1-\sqrt{5}}{2}\right),\; 
\left(1:\tfrac{\pm 1 +\sqrt{5}}{2},\;1:\tfrac{-1+\sqrt{5}}{2}\right), \; \left(1:\tfrac{\pm 1 -\sqrt{5}}{2},\;1:\tfrac{-1-\sqrt{5}}{2}\right).
$$
 Let
$$
F_1:=u^2z^4 - xyz^3w + (x^2+y^2-3u^2)z^2w^2 -xyzw^3 +u^2w^4=G_1 + H_1 u^2
$$
 be the decomposition of $F_1$
 in terms of the power of $u$,
 where
$$
 G_1=-xyz^3w+(x^2+y^2)z^2w^2-xyzw^3, \quad 
 H_1=z^4-3z^2w^2+w^4.
$$
\noindent
 Then the image $\Im(\vphi_1)$ of $\vphi_1$ is decomposed into
 three subsets
$$
 \Im(\vphi_1)=\vphi_1(U_1)=\{G_1=H_1=0\} \sqcup \{G_1=0,\;H_1\neq 0\}\sqcup \{G_1\neq 0,\;H_1\neq 0\}.
$$
 \noindent
 We can characterize these three subsets as follows:
 For any point $(x:y,\;z:w) \in \Im(\vphi_1)$,
 the fiber of $\vphi_1$ at $(x:y,\;z:w)$ is an infinite set
 if and only if
 $G_1(x,y,z,w)=H_1(x,y,z,w)=0$;
 the fiber of $\vphi_1$ at $(x:y,\;z:w)$ consists of one point
 if and only if
 $G_1(x,y,z,w)= 0$ and $H_1(x,y,z,w)\neq 0$;
 the fiber of $\vphi_1$ at $(x:y,\;z:w)$ consists of two points
 if and only if
 $G_1(x,y,z,w)\neq 0$ and $H_1(x,y,z,w)\neq 0$.
% Note that $L_1$ consists of all the points satisfying
% $g_1(x,y,z,w)=h_1(x,y,z,w)=0$, and
 For $(z:w) \in \bP^1$, we see that
 $H_1(z,w)=0$ if and only if
 $(z:w)=\left(1:\pm \frac{1 \pm \sqrt{5}}{2}\right)$.
 Then it is easy to see that $L_1$ is equal to
 the set of points satisfying
 $G_1(x,y,z,w)=H_1(x,y,z,w)=0$.
 This means each point of $L_1$ has an infinite fiber which is
% isomorphic to the affine line $\bA^1:=\bA^1_{\bC}$.
 isomorphic to the affine line $\bA^1$.
% Indeed we have
%%
%\begin{align*}
% \text{Fiber of }\left\{\left(1:\tfrac{\pm 1 +\sqrt{5}}{2},\;1:\tfrac{1+\sqrt{5}}{2}\right)\right\} \text{ in }U_1
%&= \left\{\left(1:\tfrac{\pm 1 +\sqrt{5}}{2}: u,\;1:\tfrac{1+\sqrt{5}}{2}\right)\right\} \;\isom\; \bA^1, \\
% \text{Fiber of }\left\{\left(1:\tfrac{\pm 1 -\sqrt{5}}{2},\;1:\tfrac{1-\sqrt{5}}{2}\right)\right\} \text{ in }U_1
%&= \left\{\left(1:\tfrac{\pm 1 -\sqrt{5}}{2}: u,\;1:\tfrac{1-\sqrt{5}}{2}\right)\right\} \;\isom\; \bA^1, \\
% \text{Fiber of }\left\{\left(1:\tfrac{\pm 1 +\sqrt{5}}{2},\;1:\tfrac{-1+\sqrt{5}}{2}\right)\right\} \text{ in }U_1
%&= \left\{\left(1:\tfrac{\pm 1 +\sqrt{5}}{2}: u,\;1:\tfrac{-1+\sqrt{5}}{2}\right)\right\} \;\isom\; \bA^1, \\
% \text{Fiber of }\left\{\left(1:\tfrac{\pm 1 -\sqrt{5}}{2},\;1:\tfrac{-1-\sqrt{5}}{2}\right)\right\} \text{ in }U_1
%&= \left\{\left(1:\tfrac{\pm 1 -\sqrt{5}}{2}: u,\;1:\tfrac{-1-\sqrt{5}}{2}\right)\right\} \;\isom\; \bA^1.
%\end{align*}
\noindent 
 Hence $\chi(L_1)=8$ and $\chi(\vphi^{-1}_1(L_1))=8$.
 Since $G_1=zw(xz-yw)(xw-yz)$,
 the set $B_1:=V(g_1) \subset \bP^1\times\bP^1$ is decomposed
 into $4$ subsets
\begin{align*}
 B_{11}&=V(z)=\bP^1 \times \{(0:1)\} \subset \bP^1\times\bP^1, \\
 B_{12}&=V(w)=\bP^1 \times \{(1:0)\}, \\
 B_{13}&=V(xz-yw)=\{(1:y,\; y:1),\;(0:1,\; 1:0)\} \;\isom\; \bP^1, \\
 B_{14}&=V(xw-yz)=\{(1:y,\; 1:y),\;(0:1,\; 0:1)\} \;\isom\; \bP^1.
\end{align*}
 Immediately we have
\begin{align*}
 B_{11} \cap B_{12} &=\emptyset,
 \quad B_{11}\cap B_{13}=\{(1:0,\; 0:1)\},
 \quad B_{11} \cap B_{14} = \{(0:1,\; 0:1)\}, \\
 B_{12} \cap B_{13} &= \{(0:1,\; 1:0)\},
 \quad B_{12} \cap B_{14} =\{(1:0,\;1:0)\}, \\
 B_{13} \cap B_{14} &=\{(1:1,\; 1:1),\; (1:-1,\; -1:1)\},\\
 B_{11} \cap B_{12} \cap B_{13} &=B_{11} \cap B_{12} \cap B_{14}=B_{11}\cap B_{13}\cap B_{14}=\emptyset, \\
 B_{12} \cap B_{13} \cap B_{14} &=\emptyset,
 \quad B_{11}\cap B_{12} \cap B_{13} \cap B_{14} = \emptyset.
\end{align*}
\noindent
 Hence we can compute the Euler characteristic $\chi(B_1)$:
\begin{align*}
\chi(B_1)&=\chi(B_{11}\cup B_{12} \cup B_{13} \cup B_{14})\\
         &=\chi(B_{11}) + \chi(B_{12}\cup B_{13} \cup B_{14})
          -\chi((B_{11}\cap B_{12}) \;\cup\; (B_{11}\cap B_{13}) \;\cup\; (B_{11}\cap B_{14})) \\
%         &=\cdots \\
         &=\chi(B_{11})+\chi(B_{12})+\chi(B_{13})+\chi(B_{14})
         -\chi(B_{11}\cap B_{12})-\chi(B_{11}\cap B_{13})-\chi(B_{11}\cap B_{14}) \\
         &\text{  }-\chi(B_{12}\cap B_{13})-\chi(B_{12}\cap B_{14})-\chi(B_{13}\cap B_{14}) \\
         &=2+2+2+2-0-1-1-1-1-2=2.
\end{align*}
\noindent
 Thus we have
\begin{align*}
 \chi(U_1) &= 2 \chi((\bP^1 \times \bP^1) \ssm (B_1 \sqcup Q_1))
            + \chi(B_1 \ssm L_1) + \chi(\vphi^{-1}_1(L_1)) \\
           &= 2\chi(\bP^1 \times \bP^1) -\chi(B_1) -2\chi(Q_1)
            - \chi(L_1) + \chi(\vphi^{-1}_1(L_1))\\
           &= (2\times 4) -2 -(2\times 0) -8 + 8= 6.\\
 \chi(S_1) &= \chi(U_1) + \chi(P_1)=\chi(U_1) + 4 = 10.
\end{align*}
\noindent
 Let $\tilde{S}_1$ be the desingularization of $S_1$
 by blowing up at six singular points.
 Each fiber is a smooth conic curve inside $\tilde{S}_1$, which is isomorphic to $\bP^1$.
 Thus we have
$$
\chi(\tilde{S}_1)= \chi(S_1 \ssm S_{1,\sing}) + 6 \chi(\bP^1)
= 10 -6 + 12 =16.
$$
\noindent 
 Remember that
% $\tilde{S}_1$ is a conic bundle (with singular fibers).
% Hence this is birationally equivalent to a rational surface.
 $\tilde{S}_1$ is isomorphic to
% a projective plane blown up at $n$ points.
 $\bP^1\times\bP^1$ blown up $n$ times.
% That means $\chi(\tilde{S}_1)=\chi(\bP^2) + n=n+3$, which implies
 That means $\chi(\tilde{S}_1)=\chi(\bP^1\times\bP^1) + n=n+4$,
 which implies $n= 12$.

%\subsubsection{$6^2_3$ case}
\subsection{$6^2_3$ case}

 We can compute $\chi(S_2)$ by the same way as in the $6^2_2$ case,
 therefore we omit the details.
 Let $\vphi_2 : S_2 \hr \bP^2 \times\bP^1 \dra \bP^1 \times \bP^1$
 be the rational map defined by $(x:y:u,\;z:w) \mapsto (x:y,\;z:w)$.
 This is not defined at three points $(0:0:1,\;0:1)$, $(0:0:1,\;1:1)$
 and $(0:0:1,\;1:-1)$.
 Let $P_2$ be the set of those three points and put $U_2:=S_2 \ssm P_2$.
 The image $\Im(\vphi_2)$ of $U_2$ is
 $\bP^1 \times \bP^1 \ssm Q_2$,
 where
\begin{align*}
  Q_2
   &= \bP^1 \times \{(0:1)\} \ssm \left\{(1:0,\;0:1),\; (0:1,\;0:1)\right\} \\
     &\sqcup  \bP^1 \times \{(1:1)\} \ssm \left\{(1:1,\;1:1)\right\} \\
     &\sqcup  \bP^1 \times \{(1:-1)\} \ssm \left\{(1:-1,\;1:-1)\right\}.
\end{align*}
\noindent 
 Hence $\chi(Q_2)=2$.
 Let
$$
F_2=u^2z^3-xyz^2w+(x^2+y^2-u^2)zw^2-xyw^3=G_2 + H_2 u^2
$$
 be the decomposition of $F_2$
 in terms of the power of $u$,
 where
$$
 G_2=-xyz^2w+(x^2+y^2)zw^2-xyw^3, \quad
 H_2=z(z^2-w^2).
$$
 For $(z:w) \in \bP^1$, it is easy to check that
 $H_2(z,w)=0$ if and only if
% $(z:w)=(0:1),\; (1:1), \; (1:-1)$.
 $(z:w)=(0:1),\; (1:1)$ or $(1:-1)$.
 Let $L_2$ be the subset of $\bP^1\times \bP^1$ which consists of
 the following four points
$$
 \left(1:0,\;0:1\right),\; \left(0:1,\;0:1\right),\; \left(1:1,\;1:1\right),\;\left(1:-1,\;1:-1\right).
$$
 Then we see that $L_2=\{G_2=H_2=0\}$
 as in the $6^2_2$ case.
 Therefore each point of $L_2$ has an infinite fiber isomorphic to
 the affine line $\bA^1$.
 Hence we have $\chi(L_2)=4$ and $\chi(\vphi^{-1}_2(L_2))=4$.
%
%\begin{align*}
% \text{Fiber of }\left(1:0,\;0:1\right) \text{ in }U_2
%&= \left\{\left(1:0: u,\;0:1\right)\right\} \;\isom\; \bA^1, \\
% \text{Fiber of }\left(0:1,\;0:1\right) \text{ in }U_2
%&= \left\{\left(0:1: u,\;0:1\right)\right\} \;\isom\; \bA^1, \\
% \text{Fiber of }\left(1:1,\;1:1\right) \text{ in }U_2
%&= \left\{\left(1:1: u,\;1:1\right)\right\} \;\isom\; \bA^1, \\
% \text{Fiber of }\left(1:-1,\;1:-1\right) \text{ in }U_2
%&= \left\{\left(1:-1: u,\;1:-1\right)\right\} \;\isom\; \bA^1.
%\end{align*}
\noindent 
%
%#Tr(R)-2 
%L:=(x^2+y^2+z^2-x*y*z-4)*(z^2-x*y*z+x^2+y^2-3)^2*(z*x^2-x*y-x*y*z^2+z*y^2-z+z^3)^2; 
%#Tr(Ra)-Tra 
%M:=x*(x^2+y^2+z^2-x*y*z-4)*(z^2-x*y*z+x^2+y^2-3)^2*(z*x^2-x*y-x*y*z^2+z*y^2-z+z^3)^2; 
%#Tr(Rb)-Tr(b) 
%#N:=-(z^2-x*y*z+x^2+y^2-3)*(z*x^2-x*y-x*y*z^2+z*y^2-z+z^3)*(x^3-2*x+x*y^2-2*x^2*y*z+3*y*z-y^3*z+x*y^2*z^2+z^2*x-z^3*y)*(x^2+y^2+z^2-x*y*z-4); 
%#N:=z^2-x*y*z+x^2+y^2-3; 
%#N:=z*x^2-x*y-x*y*z^2+z*y^2-z+z^3; 
%#N:=x^3-2*x+x*y^2-2*x^2*y*z+3*y*z-y^3*z+x*y^2*z^2+z^2*x-z^3*y; 
%N:=x^2+y^2+z^2-x*y*z-4; 
%
 Since $G_2=w(xz-yw)(xw-yz)$,
 the set $B_2:=V(g_2) \subset \bP^1\times\bP^1$ is decomposed
 into the following three subsets
\begin{align*}
 B_{21}&=V(w)=\bP^1 \times \{(1:0)\} \subset \bP^1\times\bP^1, \\
 B_{22}&=V(xz-yw)=\{(1:y,\; y:1),\;(0:1,\; 1:0)\} \;\isom\; \bP^1, \\
 B_{23}&=V(xw-yz)=\{(1:y,\; 1:y),\;(0:1,\; 0:1)\} \;\isom\; \bP^1, \\
%\end{align*}
%\begin{align*}
 B_{21} \cap B_{22} &=\{(0:1,\; 1:0)\}, \quad B_{21}\cap B_{23}=\{(1:0,\;1:0)\}, \\
 B_{22} \cap B_{23} &= \{(1:1,\; 1:1),\; (1:-1,\; -1:1)\}, \quad  B_{21} \cap B_{22} \cap B_{23} =\emptyset.
\end{align*}
\noindent
 Hence we have
\begin{align*}
\chi(B_2)&=\chi(B_{21}\cup B_{22} \cup B_{23})\\
         &=\chi(B_{21}\cup B_{22})+\chi(B_{23})
          -\chi(B_{21}\cap B_{23} \;\cup\; B_{22}\cap B_{23}) \\
         &=\chi(B_{21})+\chi(B_{22})+\chi(B_{23})
         -\chi(B_{21}\cap B_{22})-\chi(B_{21}\cap B_{23})-\chi(B_{22}\cap B_{23})
         + \chi(B_{21}\cap B_{22} \cap B_{23})\\
         &=2+2+2-1-1-2+0=2.
\end{align*}
\noindent
 Thus we can compute $\chi(S_2)$ as follows.
\begin{align*}
 \chi(U_2) &= 2 \chi(\bP^1 \times \bP^1 \ssm (B_2 \sqcup Q_2))
            + \chi(B_2 \ssm L_2) + \chi(\vphi^{-1}_2(L_2)) \\
           &= 2\chi(\bP^1 \times \bP^1) -\chi(B_2) -2\chi(Q_2)
            - \chi(L_2) + \chi(\vphi^{-1}_2(L_2))\\
           &= 2\times 4 -2 -(2\times 2) - 4 + 4=2.\\
 \chi(S_2) &= \chi(U_2) + \chi(P_2)=\chi(U_2) + 3=5.
\end{align*}

\noindent 
 We have already seen in \refsec{sec:ProjModel} that
% $\chi(\tilde{S}_2) = \chi(S_2) + 6 = 6 + 6 = 12$.
 $\chi(\tilde{S}_2) = \chi(S_2) + 8 = 5 + 8 = 13$.
% The desingularization
% $\tilde{S}_2$ is birational to
% the conic bundle $S_2$ (with singular fibers).
 Hence $\tilde{S}_2$ is isomorphic to $\bP^1\times\bP^1$
% blown up at $13-4= 9$ points.
 blown up $13-4= 9$ times.

%%%%%%%%%%%%%%%%%%%%%%%%%%%%%%%%%%%%%%%%%%%%%
%%%%%%%%%%%%%%% bibliography %%%%%%%%%%%%%%%
%%%%%%%%%%%%%%%%%%%%%%%%%%%%%%%%%%%%%%%%%%

%\bibliographystyle{amsplain}
%\bibliography{arithbridgelink}

\begin{thebibliography}{10}

\bibitem{Beauville}
A.\ Beauville, \emph{Complex algebraic surfaces}, second ed., London
  Math.\ Soc.\ Student Texts, \textbf{34}, Cambridge Univ.\ Press,
 1996.

\bibitem{Burde}
G.\ Burde and H.\ Zieschang, \emph{Knots}, second ed., de Gruyter
  Studies in Math. \textbf{5}.

\bibitem{CS}
M.\ Culler and P.\ Shalen, \emph{Varieties of group representations and
  splittings of {$3$}-manifolds}, Ann.\ Math.\ (2) \textbf{117} (1983), 109--146.

\bibitem{GMM2}
F.~W. Gehring, C.~Maclachlan, and G.~J. Martin, \emph{Two-generator arithmetic
  {K}leinian groups. {II}}, Bull.\ London Math.\ Soc.\ \textbf{30} (1998), 258--266.


\bibitem{GM}
F.~Gonz{\'a}lez-Acu{\~n}a and Jos{\'e}~Mar{\'{\i}}a Montesinos-Amilibia,
  \emph{On the character variety of group representations in {${\rm SL}(2,{\bf
  C})$} and {${\rm PSL}(2,{\bf C})$}}, Math.\ Z.\ \textbf{214} (1993), 627--652.

%
%\bibitem{GriHar}
%P.\ Griffiths and J.\ Harris, \emph{Principles of algebraic geometry},
%  Wiley-Interscience, 1978, Pure and Applied
%  Mathematics.
%

\bibitem{Hartshorne}
R.\ Hartshorne, \emph{Algebraic geometry}, Springer-Verlag, 1977,
  Grad.\ Texts in Math., \textbf{52}.

\bibitem{EmilyLandes}
E.\ Landes, \emph{Identifying the canonical component for the {W}hitehead
  link}, Math.\ Res.\ Lett.\ \textbf{18} (2011), 715--731.

\bibitem{W-Q}
W.\ Li and Q.\ Wang, \emph{An {${\rm SL}_2(\Bbb C)$} algebro-geometric
  invariant of knots}, International J.\ Math.\ \textbf{22} (2011),  1209--1230.


\bibitem{MPL}
M.\ Macasieb, K.\ Petersen, and R.\ van Luijk, \emph{On
  character varieties of two-bridge knot groups}, Proc.\ London Math.\ Soc.\ \textbf{103} (2011), 473--507.

\bibitem{BookMaRe}
C.\ Maclachlan and A.~W.\ Reid, \emph{The arithmetic of hyperbolic
  3-manifolds}, Grad.\ Texts in Math., \textbf{219}, Springer-Verlag, New
  York, 2003.

\bibitem{Vicente}
V.\ Mu{\~n}oz, \emph{The {${\rm SL}(2,\Bbb C)$}-character varieties of
  torus knots}, Rev.\ Mat.\ Complut.\ \textbf{22} (2009), 489--497.

\bibitem{RileyNonab}
R.\ Riley, \emph{Nonabelian representations of {$2$}-bridge knot groups},
  Quart.\ J.\ Math.\ Oxford Ser.\ (2) \textbf{35} (1984), 191--208.


\bibitem{ShalenHandbook}
P.\ Shalen, \emph{Representations of 3-manifold groups}, Handbook of
  geometric topology, North-Holland, Amsterdam, 2002, 955--1044.

\bibitem{AdvancedSilverman}
J.\ H.\ Silverman, \emph{Advanced topics in the arithmetic of elliptic
  curves}, Grad.\ Texts in Math., \textbf{151}, Springer-Verlag, New York,
  1994.


\end{thebibliography}

\providecommand{\bysame}{\leavevmode\hbox to3em{\hrulefill}\thinspace}
\providecommand{\MR}{\relax\ifhmode\unskip\space\fi MR }
\providecommand{\MRhref}[2]{%
  \href{http://www.ams.org/mathscinet-getitem?mr=#1}{#2}
}
\providecommand{\href}[2]{#2}

%%%%%%%%%%%%%   Address   %%%%%%%%%%%%%%%%%%%%%%%%

%\vspace{1cm}
\vspace{0.5cm}

\noindent
Shinya Harada\\
School of Mathematics\\
Korea Institute for Advanced Study (KIAS)\\
Hoegiro 87,
%207-43 Cheongnyangni 2-dong\\
 Dongdaemun-gu, Seoul 130-722\\
 Republic of Korea\\
{\tt harada@kias.re.kr}

\end{document}